\documentclass{amsart}

\usepackage{amssymb,amsmath,amsthm}

\usepackage{graphicx}

\usepackage{xcolor}
\usepackage{mathtools}
\usepackage{microtype}
\usepackage{subfig} 
\newtheorem{theorem}{Theorem}[section]

\newtheorem*{proposition*}{Proposition}
\newtheorem{proposition}[theorem]{Proposition}
\newtheorem{corollary}[theorem]{Corollary}

\theoremstyle{definition}
\newtheorem{definition}[theorem]{Definition}

\theoremstyle{remark}
\newtheorem{remark}[theorem]{Remark}

\numberwithin{equation}{section}

\newcommand{\rr}[1]{#1}

\newcommand{\R}{\mathbb{R}}
\newcommand{\K}{\mathbb{K}}
\newcommand{\C}{\mathbb{C}}
\newcommand{\V}{\mathbb{V}}
\newcommand{\mC}{\mathcal{C}}
\newcommand{\bC}{\mathbf{C}}
\newcommand{\bX}{\mathbf{X}}
\newcommand{\bY}{\mathbf{Y}}
\newcommand{\bZ}{\mathbf{Z}}
\newcommand{\frob}{\mathsf{F}}
\newcommand{\abs}[1]{{\left\lvert #1 \right\rvert}}

\DeclareMathOperator{\rank}{rank}

\DeclareMathOperator{\App}{App}

\begin{document}

\title[Orthogonal tensors and rank-1 approximation ratio]
{On orthogonal tensors and best rank-one approximation ratio}


\author{Zhening Li}
\address{Department of Mathematics, University of Portsmouth, Portsmouth, Hampshire PO1 3HF, United Kingdom}
\curraddr{}
\email{zheningli@gmail.com}
\thanks{}

\author{Yuji Nakatsukasa}
\address{Mathematical Institute, University of Oxford, Oxford OX2 6GG, United Kingdom
}
\curraddr{}
\email{nakatsukasa@maths.ox.ac.uk}
\thanks{YN is supported by JSPS as an Overseas Research Fellow}

\author{Tasuku Soma}
\address{
Graduate School of Information Science \& Technology,
            University of Tokyo,
            7-3-1 Hongo, Bunkyo-ku, Tokyo, Japan }
\curraddr{}
\email{tasuku$\_$soma@mist.i.u-tokyo.ac.jp}
\thanks{TS is supported by JST CREST Grant Number JPMJCR14D2, Japan}

\author{Andr\'e Uschmajew}
\address{     Hausdorff Center for Mathematics \& Institute for Numerical Simulation, University of Bonn, 53115 Bonn, Germany}
\curraddr{Max Planck Institute for Mathematics in the Sciences, 04103 Leipzig, Germany}
\email{uschmajew@mis.mpg.de}
\thanks{}


\date{}

\dedicatory{}


\begin{abstract}

As is well known, the smallest possible ratio between the spectral norm and the Frobenius norm of an $m \times n$ matrix with $m \le n$ is $1/\sqrt{m}$ and is (up to scalar scaling) attained only by matrices having pairwise orthonormal rows. In the present paper, the smallest possible ratio between spectral and Frobenius norms of $n_1 \times \dots \times n_d$ tensors of order $d$, also called the best rank-one approximation ratio in the literature, is investigated. The exact value is not known for most configurations of $n_1 \le \dots \le n_d$. Using a natural definition of orthogonal tensors over the real field (resp., unitary tensors over the complex field), it is shown that the obvious lower bound $1/\sqrt{n_1 \cdots n_{d-1}}$ is attained if and only if a tensor is orthogonal (resp., unitary) up to scaling. Whether or not orthogonal or unitary tensors exist depends on the dimensions $n_1,\dots,n_d$ and the field. A connection between the (non)existence of real orthogonal tensors of order three and the classical Hurwitz problem on composition algebras can be established: existence of orthogonal tensors of size $\ell \times m \times n$ is equivalent to the admissibility of the triple $[\ell,m,n]$ to the Hurwitz problem. Some implications for higher-order tensors are then given. For instance, real orthogonal $n \times \dots \times n$ tensors of order $d \ge 3$ do exist, but only when $n = 1,2,4,8$. In the complex case, the situation is more drastic: unitary tensors of size $\ell \times m \times n$ with $\ell \le m \le n$ exist only when $\ell m \le n$. Finally, some numerical illustrations for spectral norm computation are presented.
\end{abstract}

\maketitle


\section{Introduction}

Let $\K$ be $\R$ or $\C$. Given positive integers $d \ge 2$ and $n_1,\dots,n_d$, we consider the tensor product
\[
\V = V^1 \otimes \dots \otimes V^d
\]
of Euclidean $\K$-vector spaces $V^1,\dots,V^d$ of dimensions $\dim(V^\mu) = n_\mu$, $\mu=1,\dots,d$. The space $\V$ is generated by the set of \emph{elementary} (or \emph{rank-one}) tensors
\[
\mC_{1} = \{ u^1 \otimes \dots \otimes u^d \colon u^1 \in V^1, \dots, u^d \in V^d \}.
\]
In general, elements of $\V$ are called tensors. The natural inner product on the space $\V$ is uniquely determined by its action on decomposable tensors via
\[
\langle u^1 \otimes \dots \otimes u^d, v^1 \otimes \dots \otimes v^d \rangle_\frob = \prod_{\mu=1}^d \langle u^\mu, v^\mu \rangle_{V^i}.
\]	
This inner product is called the Frobenius inner product, and its induced norm is called the Frobenius norm, denoted by $\| \cdot \|_\frob$.

\subsection{Spectral norm and best rank-one approximation}

The \emph{spectral norm} (also called \emph{injective norm}) of a tensor $\bX \in \V$ is defined as
\begin{equation}\label{eq: def spectral norm}
\| \bX \|_2 = \max_{\substack{\bY \in \mC_1 \\ \| \bY \|_\frob = 1 }} \abs{\langle \bX , \bY \rangle_\frob} = \max_{\| u^1 \|_{V^1} = \dots = \| u^d \|_{V^d} = 1} \abs{\langle \bX , u^1 \otimes \dots \otimes u^d \rangle_\frob}.
\end{equation}
Note that the second $\max$ is achieved by some $u^1 \otimes \dots \otimes u^d$ since the spaces $V^\mu$'s are finite dimensional. Hence the first $\max$ is also achieved. Checking the norm properties is an elementary exercise.

Since the space $\V$ is finite dimensional, the Frobenius norm and spectral norm are equivalent. It is clear from the Cauchy--Schwarz inequality that
\[
\| \bX \|_2 \le \| \bX \|_\frob.
\]
The constant one in this estimate is optimal, since equality holds for elementary tensors.

For the reverse estimate, the maximal constant $c$ in
\[
c \| \bX \|_\frob \le \| \bX \|_2
\]
is unknown in general and may depend not only on $d$, $n_1,\dots,n_d$ but also on $\K$. Formally, the optimal value is defined as
\begin{equation}\label{eq: definition of App}
\App(\V) \equiv \App_d(\K ; n_1,\dots,n_d) \coloneqq \min_{\bX \neq 0} \frac{\| \bX \|_2}{\| \bX \|_\frob} =  \min_{\| \bX \|_\frob = 1} \| \bX \|_2.
\end{equation}
Note that by continuity and compactness, there always exists a tensor $\bX$ achieving the minimal value.

The task of determining the constant $\App(\V)$ was posed by Qi~\cite{Qi2011}, who called it the \emph{best-rank one approximation ratio} of the tensor space $\V$. This terminology originates from the important geometrical fact that the spectral norm of a tensor measures its approximability by elementary tensors. To explain this, we first recall that $\mC_1$, the set of elementary tensors, is closed and hence
every tensor $\bX$ admits a best approximation (in Frobenius norm) in $\mC_1$. Therefore, the problem of finding $\bY_1 \in \mC_1$ such that
\begin{equation}\label{eq: best rank-one approximation problem}
\| \bX - \bY_1 \|_\frob = \inf_{\bY \in \mC_1} \| \bX - \bY \|_\frob
\end{equation}
has at least one solution. Any such solution is called a \emph{best rank-one approximation to $\bX$}. The relation between the best rank-one approximation of a tensor and its spectral norm is given as follows.
\begin{proposition}\label{prop}
A tensor $\bY_1 \in \mC_1$ is a best rank-one approximation to $\bX \neq 0$ if and only if the following holds:
\[
\| \bY_1 \|_\frob = \left\langle \bX , \frac{\bY_1}{\| \bY_1 \|_\frob} \right\rangle_\frob  = \| \bX \|_2.
\]
Consequently,
\begin{equation}\label{eq: rank-one error}
\| \bX - \bY_1 \|_\frob^2 = \| \bX \|_\frob^2 - \| \bX \|_2^2.
\end{equation}
\end{proposition}
The original reference for this observation is hard to trace back; see, e.g.,~\cite{KoldaBader2009}. It is now considered a folklore. The proof is easy from some least-square argument based on the fact that $\mC_1$ is a $\K$-double cone, i.e., $\bY \in \mC_1$ implies $t \bY \in \mC_1$ for all $t \in \K$.

By Proposition~\ref{prop}, the rank-one approximation ratio $\App(\V)$ is equivalently seen as the worst-case angle between a tensor and its best rank-one approximation:
\[
\App(\V) = \min_{\bX \neq 0} \frac{\abs{\langle \bX , \bY_1 \rangle_\frob}}{\| \bX \|_\frob \cdot \| \bY_1 \|_\frob},
\]
where $\bY_1\in\mC_1$ depends on $\bX$. As an application, the estimation of $\App(\V)$ from below has some important implications for the analysis of truncated steepest descent methods for tensor optimization problems; see~\cite{Uschmajew2015}.

Combining~\eqref{eq: definition of App} and~\eqref{eq: rank-one error} one obtains
\begin{equation*}\label{eq: useful formulation}
\App(\V)^2 =  1 - \max_{\| \bX \|_\frob = 1} \min_{\bY \in \mC_1} \| \bX - \bY \|_\frob^2.
\end{equation*}

\subsection{Nuclear norm}

The nuclear norm (also called \emph{projective norm}) of a tensor $\bX \in \V$ is defined as
\begin{equation}\label{eq: definition of nuclear norm}
 \| \bX \|_* =  \inf \left\{ \sum_k \| \bZ_k \|_\frob \colon \bX = \sum_k \bZ_k \text{ with } \bZ_k \in \mC_1 \right\}.
\end{equation}
It is known \rr{(see, e.g., \cite[Thm.~2.1]{CobosKuehnPeetre1992})}
that the dual of the nuclear norm is the spectral norm (in tensor products of Banach spaces the spectral norm is usually defined in this way):
\[
 \| \bX \|_2 = \max_{\| \bY \|_* = 1} \abs{\langle \bX, \bY \rangle_\frob }.
\]
By a classic duality principle in finite-dimensional spaces (see, e.g.,~\cite[Thm.~5.5.14]{HornJohnson1985}), the nuclear norm is then also the dual of the spectral norm:
\[
 \| \bX \|_* = \max_{\| \bY \|_2 = 1} \abs{\langle \bX, \bY \rangle_\frob }.
\]
\rr{It can be shown that this remains true in tensor products of infinite-dimensional Hilbert spaces~\cite[Thm.~2.3]{CobosKuehnPeetre1992}}.

Either one of these duality relations immediately implies that
\begin{equation}\label{eq: duality estimate}
 \| \bX \|_\frob^2 \le \|\bX\|_2 \| \bX \|_*.
\end{equation}
In particular, $\| \bX \|_\frob \le \| \bX \|_*$ and equality holds if and only if $\bX$ is an elementary tensor.

Regarding the sharpest norm constant for an inequality $\| \bX \|_* \le c \| \bX\|_\frob$, it is shown in~\cite[Thm.~2.2]{DerksenFriedlandLimWang2017} that
\begin{equation}\label{eq: equivalence for spectral and nuclear}
 \max_{\bX \neq 0} \frac{\| \bX \|_*}{\| \bX \|_\frob} = \left( \min_{\bX \neq 0} \frac{\| \bX \|_2}{\| \bX \|_\frob}  \right)^{-1} = \frac{1}{\App(\V)}.
\end{equation}
This is a consequence of the duality of the nuclear and spectral norms. Moreover, the extremal values for both ratios are achieved by the same tensors $\bX$.

Consequently, determining the exact value of $\max_{\bX \neq 0} \| \bX \|_* / \| \bX \|_\frob$ is equivalent to determining $\App(\V)$. An obvious bound that follows from the definition~\eqref{eq: definition of nuclear norm} and the Cauchy--Schwarz inequality is
\begin{equation}\label{eq: trivial bound for nuclear norm}
 \frac{\| \bX \|_*}{\| \bX \|_\frob} \le \sqrt{\rank_\bot(\bX)} \le \sqrt{\min_{\nu =1,\dots,d} \prod_{\mu \neq \nu} n_\mu},
\end{equation}
where $\rank_\bot(\bX)$ is the orthogonal rank of $\bX$; cf. section~\ref{sec: orthogonal rank}.

\subsection{Matrices}\label{sec: matrices}

It is instructive to inspect the matrix case. In this case, it is well known that
\begin{equation}\label{eq: App for matrices}
\App_2(\K;m,n) = \frac{1}{\sqrt{\min(m,n)}}.
\end{equation}
In fact, let $\bX \in \K^{m \times n}$ have $\rank(\bX) = R$ and
\[
\bX = \sum_{k=1}^{R} \sigma_k u_k \otimes v_k
\]
be a singular value decomposition (SVD) with orthonormal systems $\{u_1,\dots,u_R\}$ and $\{v_1,\dots,v_R\}$, and $\sigma_1 \ge \sigma_2 \ge \dots \ge \sigma_R > 0$. Then by a well-known theorem~\cite[Thm.~2.4.8]{golubbook4th} the best rank-one approximation of $\bX$ in Frobenius norm is given by
\[
\bX_1 = \sigma_1 u_1 \otimes v_1,
\]
producing an approximation error
\[
\| \bX - \bX_1 \|_\frob^2 = \sum_{k=2}^R \sigma_k^2.
\]
The spectral norm is
\[
\| \bX \|_2 = \| \bX_1 \|_\frob = \sigma_1 \ge \frac{\| \bX \|_\frob}{\sqrt{R}}.
\]
Here equality is attained only for a matrix with $\sigma_1 = \dots = \sigma_R = \frac{\| \bX \|_\frob}{\sqrt{R}}$. Obviously,~\eqref{eq: App for matrices} follows when $R = \min(m,n)$. Hence, assuming $m \le n$, we see from the SVD that a matrix $\bX$ achieving equality satisfies $\bX \bX^H = \frac{\| \bX \|_\frob^2}{m} I_{m}$ with $I_m$ the $m \times m$ identity matrix, that is, $\bX$ is a multiple of a matrix with pairwise orthonormal rows.

Likewise it holds for the nuclear norm of a matrix that
\[
 \| \bX \|_* = \sum_{k=1}^R \sigma_k \le \sqrt{\min(m,n)} \cdot \| \bX \|_\frob,
\]
and equality is achieved (in the case $m \le n$) if and only if $\bX$ is a multiple of a matrix with pairwise orthonormal rows.

\subsection{Contribution and outline}

As explained in section~\ref{sec: orthogonal rank} below, it is easy to deduce the ``trivial'' lower bound
\begin{equation}\label{eq: trivial lower bound}
\App_d(\K; n_1, \dots,n_d) \ge \frac{1}{\sqrt{\min_{\nu = 1,\dots,d} \prod_{\mu \neq \nu} n_\mu}}
\end{equation}
for the best rank-one approximation ratio of a tensor space. From~\eqref{eq: App for matrices} we see that this lower bound is sharp for matrices for any $(m,n)$ and is attained only at matrices with pairwise orthonormal rows or columns, or their scalar multiples (in this paper, with a slight abuse of notation, we call such matrices \emph{orthogonal} when $\K=\R$ (resp., \emph{unitary} when $\K=\C$)). A key goal in this paper is to generalize this fact to higher-order tensors.

First in section~\ref{sec: previous results} we review some characterizations of spectral norm and available bounds on the best-rank one approximation ratio.

In section~\ref{sec: orthogonal tensors} we show that the trivial rank-one approximation ratio~\eqref{eq: trivial lower bound} is achieved if and only if a tensor is a scalar multiple of an \emph{orthogonal} (resp., \emph{unitary}) tensor, where the notion of orthogonality (resp., unitarity) is defined in a way that generalizes orthogonal (resp., unitary) matrices very naturally. We also prove corresponding extremal properties of orthogonal (resp., unitary) tensors regarding the ratio of the nuclear and Frobenius norms.

We then study in section~\ref{sec: existence} further properties of orthogonal tensors, in particular focusing on their existence. Surprisingly, unlike the matrix case where orthogonal/unitary matrices exist for any $(m,n)$, orthogonal tensors often do not exist, depending on the configuration of $(n_1,\ldots,n_d)$ and the field $\K$. In the first nontrivial case $d=3$ over $\K=\mathbb{R}$,
we show that the (non)existence of orthogonal tensors is connected to the classical \emph{Hurwitz problem}. This problem has been studied extensively, and in particular a result by Hurwitz himself~\cite{Hurwitz1898} implies that an $n\times n\times n$ orthogonal tensor exists only for $n=1,2,4$, and $8$, and is then essentially equivalent to a multiplication tensor in the corresponding composition algebras on $\R^n$. These algebras are the reals ($n=1$), the complex numbers ($n=2$), the quaternions ($n=4$), and the octonions ($n=8$). We further generalize Hurwitz's result to the case $d>3$. These observations might give an impression that considering orthogonal tensors is futile. However, the situation is vastly different when the tensor is not cubical, that is, when $n_\mu$'s take different values. While a complete analysis of the (non)existence of noncubic real orthogonal tensors is largely left an open problem, we investigate this problem and derive some cases where orthogonal tensors do exist. 
When $\K=\C$, the situation turns out to be more restrictive: we show that when $d\ge3$, unitary cubic tensors do not exist unless trivially $n=1$, and noncubic ones do exist only in the trivial case of extremely ``tall'' tensors, that is, if $n_\nu \ge \prod_{\mu \neq \nu} n_\mu$ for some dimension $n_\nu$.

Unfortunately, we are currently unable to provide the exact value or sharper lower bounds on the best rank-one approximation ratio of tensor spaces where orthogonal (resp., unitary) tensors do not exist. The only thing we can conclude is that in these spaces the bound~\eqref{eq: trivial lower bound} is not sharp. For example,
\[
\App_3(\R;n,n,n) > \frac{1}{n}
\]
for all $n \neq 1,2,4,8$. However, recent results on random tensors imply that the trivial lower bound provides the correct order of magnitude, that is,
\[
\App_d(\K; n_1, \dots,n_d) = O\left( \frac{1}{\sqrt{\min_{\nu = 1,\dots,d} \prod_{\mu \neq \nu} n_\mu}}\right),
\]
at least when $\K = \R$; see section~\ref{sec: random results}.

Some numerical experiments for spectral norm computation are conducted in section~\ref{sec: accurate computation}, comparing algorithms from the Tensorlab toolbox~\cite{tensorlab} with an alternating SVD (ASVD) method proposed in~\cite[sec.~3.3]{DeLathauweretal2000b} and later in~\cite{friedland2013best}. In particular, computations for random $n \times n \times n$ tensors indicate that $\App_3(\R;n,n,n)$ behaves like $O(1/n)$.

\subsection*{Some more notational conventions}

For convenience, and without loss of generality, we will identify the space $\V$ with the space $\K^{n_1} \otimes \dots \otimes \K^{n_d} \cong \K^{n_1 \times \dots \times n_d}$ of $d$-way arrays $[\bX(i_1,\dots,i_d)]$ of size $n_1 \times \dots \times n_d$, where every $\K^{n_\mu}$ is endowed with a standard Euclidean inner product $x^H y$. This is achieved by fixing orthonormal bases in the space $V^\mu$.

In this setting, an elementary tensor has \emph{entries}
\[
[u^1 \otimes \dots \otimes u^d]_{i_1,\dots,i_d} =  u^1(i_1) \cdots u^d(i_d).
\]
The Frobenius inner product of two tensors is then
\[
\langle \bX, \bY \rangle_\frob = \sum_{i_1,\dots,i_d} \overline{\bX(i_1,\dots,i_d)} \bY(i_1,\dots,i_d).
\]
It is easy to see that the spectral norm defined below is not affected by the identification of $\K^{n_1 \times \dots \times n_d}$ and $V^1 \otimes \dots \times \otimes V^d$ in the described way. 

For readability it is also useful to introduce the notation
\[
\V^{[\mu]} \coloneqq \K^{n_1} \otimes \dots \otimes \K^{n_{\mu-1}} \otimes \K^{n_{\mu+1}} \otimes \dots \otimes \K^{n_d} \cong \K^{n_1 \times \dots \times n_{\mu-1} \times n_{\mu+1} \times \dots \times n_d},
\]
which is a tensor product space of order $d-1$. The set of elementary tensors in this space is denoted by $\mC_1^{[\mu]}$.

An important role in this work is played by slices of a tensor and their linear combinations. Formally, such linear combinations are obtained as partial contractions with vectors. We use standard notation~\cite{KoldaBader2009} for these contractions: let $\bX_{i_\mu} = \bX(\colon,\dots,\colon,i_\mu,\colon,\dots,\colon) \in \V^{[\mu]}$ for $i_\mu = 1,\dots,n_\mu$, denoting the slices of the tensor $\bX \in \K^{n_1 \times \dots \times n_d}$ perpendicular to mode $\mu$. Given $u^\mu \in \K^{n_\mu}$, the mode-$\mu$ product of $\bX$ and $u^\mu$ is defined as
\[
\bX \times_\mu u^\mu \coloneqq \sum_{i_\mu = 1}^{n_\mu} u(i_\mu) \bX_{i_\mu} \in \V^{[\mu]}.
\]
Correspondingly, partial contractions with more than one vectors are obtained by applying single contractions repeatedly, for instance,
\[
\bX \times_1 u^1 \times_2 u^2 :=  (\bX \times_2 u^2) \times_1 u^1 = (\bX \times_1 u^1) \times_1 u^2,
\]
where $\times_1 u^2$ in the last equality is used instead of $\times_2 u^2$ since the first mode of $\bX$ is vanished by $\times_1 u^1$. With this notation, we have
\begin{equation}\label{eq: recursive contraction}
\langle \bX, u^1 \otimes \dots \otimes u^d \rangle_\frob = \bX \times_1 u^1 \cdots \times_d u^d = \langle \bX \times_1 u^1, u^2 \otimes \dots \otimes u^d \rangle_\frob.
\end{equation}

\section{Previous results on best rank-one approximation ratio}\label{sec: previous results}

For some tensor spaces the best rank-one approximation ratio has been determined, most notably for $\K = \R$.

\rr{
K\"uhn and Peetre~\cite{KuehnPeetre2006} determined all values of $\App_3(\R;\ell,m,n)$ with $2 \le \ell \le m \le n \le 4$, except for $\App_3(\R;3,3,3)$. For $\ell = m=2$ it holds that
\[
 \App_3(\R;2,2,2) = \App_3(\R;2,2,3) = \App_3(\R;2,2,4) = \frac{1}{2}
\]
(the value of $\App_3(\R;2,2,2)$ was found earlier in~\cite{CobosKuehnPeetre2000}). The other values for $\ell = 2$ are
\begin{equation}\label{eq: other known values}
 \App_3(\R;2,3,3) = \frac{1}{\sqrt{5}}, \quad \App_3(\R;2,3,4) = \frac{1}{\sqrt{6}}, \quad \App_3(\R;2,4,4) = \frac{1}{\sqrt{8}},
\end{equation}
whereas for $\ell \ge 3$ it holds that
\[
 \App_3(\R;3,3,4) = \frac{1}{3}, \quad \App_3(\R;3,4,4) = \frac{1}{\sqrt{12}}, \quad \App_3(\R;4,4,4) = \frac{1}{4}.
\]
It is also stated in~\cite{KuehnPeetre2006} that
\[
 \App_3(\R;8,8,8) = \frac{1}{8},
\]
and the value $\App_3(\R;3,3,3)$ is estimated to lie between $1/\sqrt{7.36}$ and $1/\sqrt{7}$. 
}

\rr{
Note that in all cases listed above except $[2,3,3]$ and $[3,3,3]$, the naive lower bound~\eqref{eq: trivial lower bound} is hence sharp. In our paper we deduce this from the fact that the corresponding triples $[\ell,m,n]$ are admissible to the Hurwitz problem, while $[2,3,3]$ and $[3,3,3]$ are not; see section~\ref{sec: Hurwitz problem}. In fact, K\"uhn and Peetre obtained the values for $\App_3(\R;n,n,n)$ for $n=4,8$ by considering the tensors representing multiplication in the quaternion and octonion algebra, respectively, which are featured in our discussion as well; see in particular Theorem~\ref{th: Hurwitz result} and Corollary~\ref{cor: existence in small dimensions}.
}

\rr{
More general recent results by Kong and Meng~\cite{KongMeng2015} are
\begin{equation}\label{eq: kongetal}
 \App_3(\R;2,m,n) = \frac{1}{\sqrt{2m}} \quad \text{for $2 \le m \le n$ and $m$ even}
\end{equation}
and
\begin{equation}\label{eq: kongetal results}
 \App_3(\R;2,n,n) = \frac{1}{\sqrt{2n - 1}} \quad \text{for $n$ odd.}\footnote{\rr{In~\cite{KongMeng2015} it is incorrectly concluded from this that $\App_3(\R;2,m,n) = 1/\sqrt{2m - 1}$ whenever $2 \le m \le n$ and $m$ is odd. By~\eqref{eq: other known values}, this is not true for $m=3$, $n=4$ and is also false whenever $n \ge 2m$ by Proposition~\ref{prop: sharpness for tall tensors} below.}}
\end{equation}
Hence the naive bound~\eqref{eq: trivial lower bound} is sharp in the first case, but not in the second. Here, since obviously $\App_3(\R;2,m,n) \le \App_3(\R;2,m,m)$ for $m \le n$, it is enough to prove the first case for $m = n$ being even. Again, we can recover this result in section~\ref{sec: Hurwitz problem} by noting that the triple $[2,n,n]$ is always admissible to the Hurwitz problem when $n$ is even due to the classic Hurwitz--Radon formula~\eqref{eq: Radon-Hurwitz}. Admittedly, the proof of~\eqref{eq: kongetal} in~\cite{KongMeng2015} is simple enough.
}

The value $\App_d(\C;2,\dots,2)$ is of high interest in quantum information theory, where \emph{multiqubit states}, $\bX \in \C^{2\times \dots \times 2}$ with $\| \bX \|_\frob = 1$ are considered. The distance $1 - \| \bX \|_2^2$ of such a state to the set of \emph{product states}, that is, the distance to its best rank-one approximation (cf.~\eqref{eq: rank-one error}), is called the \emph{geometric measure of entanglement}. In this terminology, $1 - (\App_d(\C;2,\dots,2))^2$ is the value of the maximum possible entanglement of a multiqubit state. It is known that~\cite{CobosKuehnPeetre2000}
\[
  \App_3(\C;2,2,2) = \frac{2}{3}.
\]
This result was rediscovered later by Derksen et al.~\cite{DerksenFriedlandLimWang2017} based on the knowledge of the most entangled state in $\C^{2 \times 2 \times 2}$ due to~\cite{ChenXuZhu2010}. The authors of~\cite{DerksenFriedlandLimWang2017} have also found the value
\[
\App_4(\C;2,2,2,2) = \frac{\sqrt{2}}{3}.
\]
This confirms a conjecture in~\cite{HiguchiSudbery2000}. We can see that in these cases the trivial bound~\eqref{eq: trivial lower bound} is not sharp. In fact,~\cite{DerksenFriedlandLimWang2017} provides an estimate
\[
 \App_d(\C;2,\dots,2) \ge \left(\frac{2}{3} \right) \frac{1}{\sqrt{2^{d-3}}} = \left( \frac{4}{3} \right) \frac{1}{\sqrt{2^{d-1}}}, \quad d \ge 3,
\]
so the bound~\eqref{eq: trivial lower bound} is never sharp for multiqubits (except when $d=2$). The results in this paper imply that
\begin{equation}\label{eq: nonsharpness for higher-order complex}
 \App_d(\C;n_1,\dots,n_d) > \frac{1}{\sqrt{n_1\cdots n_{d-1}}}
\end{equation}
whenever $n_1 \le \dots \le n_d$ and $n_{d-2} n_{d-1} > n_d$ (Corollary~\ref{cor: main corollary} and Theorem~\ref{th: nonexistence in complex case}).

In the rest of this section we gather further strategies for obtaining bounds on the spectral norm and best rank-one approximation ratio. For brevity we switch back to the notation $\App(\V)$ when further specification is not relevant.

\subsection{Lower bounds from orthogonal rank}\label{sec: orthogonal rank}

Lower bounds of $\App(\V)$ can be obtained from expansion of tensors into pairwise orthogonal decomposable tensors. For any $\bX \in \V$, let $R$ be an integer such that
\begin{equation}\label{eq:orthogonal decomposition}
\bX = \sum_{k=1}^R \bZ_k
\end{equation}
with $\bZ_1,\dots,\bZ_R \in \mC_1$ being mutually orthogonal. We can assume that $\| \bZ_1 \|_\frob\geq \| \bZ_2 \|_\frob\geq\cdots \geq \| \bZ_R \|_\frob$, hence that $\| \bZ_1 \|_\frob \ge \frac{\| \bX \|_\frob}{\sqrt{R}}$, and so
\begin{equation}\label{eq: estimating by orthogonal rank}
\frac{\| \bX \|_2}{ \| \bX \|_\frob} \ge \frac{\abs{\langle \bX , \bZ_1 \rangle_\frob}}{\| \bX \|_\frob \cdot \| \bZ_1 \|_\frob} = \frac{\| \bZ_1 \|_\frob}{\| \bX \|_\frob} \ge \frac{1}{\sqrt{R}}.
\end{equation}
For each $\bX$ the smallest possible value of $R$ for which a decomposition~\eqref{eq:orthogonal decomposition} is possible is called the \emph{orthogonal rank} of $\bX$~\cite{Kolda2001}, denoted by $\rank_\bot(\bX)$. Then, it follows that
\begin{equation}\label{eq: estimate from orthogonal decomposition}
\App(\V) \ge \frac{1}{\sqrt{\rank_\bot(\bX)}} \quad \text{for all $\bX \in \V$.}
\end{equation}
A possible strategy is to estimate the maximal possible orthogonal rank of the space $\V$ (which is an open problem in general). For instance, the result~\eqref{eq: kongetal results} from~\cite{KongMeng2015} is obtained by estimating orthogonal rank.

The trivial lower bound~\eqref{eq: trivial lower bound} is obtained by noticing that every tensor can be decomposed into pairwise orthogonal elementary tensors that match the entries of the tensor in single parallel \emph{fibers}\footnote{A fiber is a subset of entries $(i_1,\dots,i_d)$ in a tensor, in which one index $i_\mu$ varies from $1$ to $n_\mu$, while other indices are kept fixed.} and are zero otherwise. Depending on the orientation of the fibers, there are $\prod_{\mu \neq \nu} n_\mu$ of them. Therefore,
\begin{equation}\label{eq: trivial bound for orthogonal rank}
 \max_{\bX \in \K^{n_1 \times \dots \times n_d}} \rank_\bot (\bX) \le \min_{\nu = 1,\dots,d} \prod_{\mu \neq \nu} n_\mu
\end{equation}
and~\eqref{eq: trivial lower bound} follows from~\eqref{eq: estimate from orthogonal decomposition}; see Figure~\ref{fig: fibers and normal form}(A).

\begin{figure}[h]
\centering
\subfloat[Orthogonal decomposition of a tensor into its longest fibers. A fiber of largest Euclidean norm provides a lower bound of the spectral norm.]{\makebox[.45\textwidth][c]{\includegraphics[width = .25\textwidth]{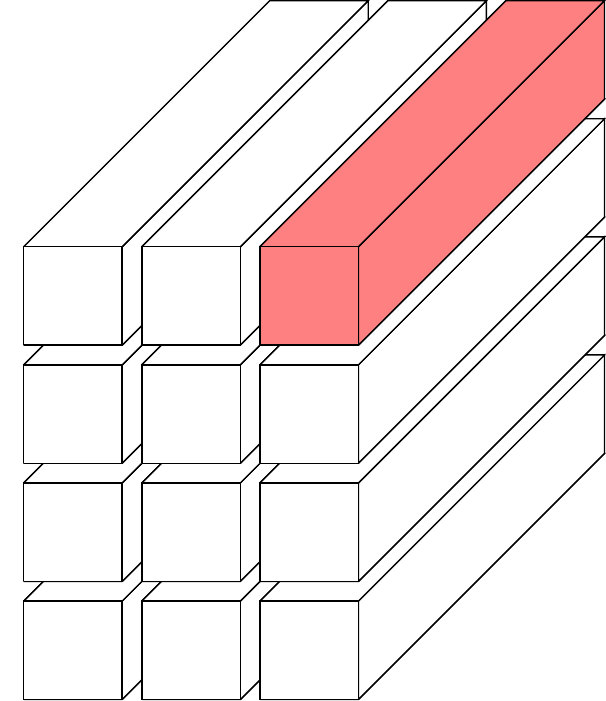}}}
\quad
\subfloat[Normal form using an orthonormal tensor product basis that includes a normalized best rank-one approximation. The red entry equals the spectral norm.]{\makebox[.45\textwidth][c]{\includegraphics[width = .25\textwidth]{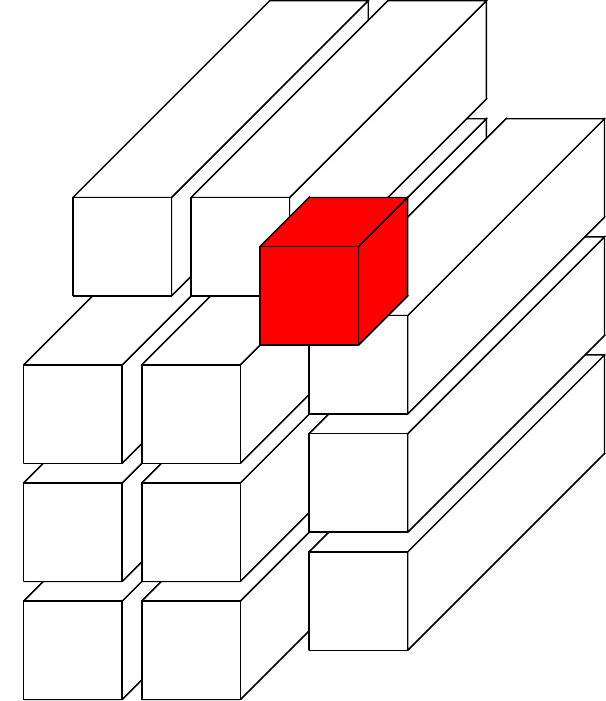}}}
\caption{Illustration of fibers and spectral normal form.}
 \label{fig: fibers and normal form}
\end{figure}

It is interesting and useful to know that after a suitable  orthonormal change of basis we can always assume that the entry $\bX(1,\dots,1)$ of a tensor $\bX \in \K^{n_1 \times \dots \times n_d}$ equals its spectral norm. In fact, let $\| \bX \|_2 \cdot (u^1_1 \otimes \dots \otimes u^d_1)$ be a best rank-one approximation of $\bX$ with $u^1_1,\dots,u^d_1$ all normalized to one. Then we can extend
$u_1^\mu$
to orthonormal bases $\{u^\mu_1, \dots, u^\mu_{n_\mu}\}$ for every $\mu$ to obtain a representation
\[
 \bX = \sum_{i_1=1}^{n_1} \cdots \sum_{i_d = 1}^{n_d} \bC(i_1,\dots,i_d) u^1_{i_1} \otimes \dots \otimes u^d_{i_d}.
\]
We may identify $\bX$ with its new coefficient tensor $\bC$; in particular, they have the same spectral norm. Since (see Proposition~\ref{prop} for the second equality)
\[
\bC(1,\dots,1) = \langle \bX, u^1_1 \otimes \dots \otimes u^d_1 \rangle_\frob = \| \bX \|_2 = \| \bC \|_2,
\]
and considering the overlap with fibers, we see that all other entries of any fiber that contains $\bC(1,\dots,1)$ must be zeros; see Figure~\ref{fig: fibers and normal form}(B). This ``spectral normal form'' of the tensor $\bX$ can be used to study uniqueness and perturbation of best rank-one approximation of tensors~\cite{JiangKong2015}. For our purposes, the following conclusion will be of interest and is immediately obtained by decomposing the tensor $\bC$ into fibers.

\begin{proposition}\label{prop: decomposition into normal form}
 Let $\bX \in \K^{n_1 \times \dots \times n_d}$. For any $\nu = 1,\dots,d$, there exists an orthogonal decomposition~\eqref{eq:orthogonal decomposition} into $R = \prod_{\mu \neq \nu} n_\mu$ mutually orthogonal elementary tensors $\bZ_k$ such that $\bZ_1$ is a best rank-one approximation of $\bX$. In particular, $\| \bZ_1 \|_\frob = \| \bX \|_2$.
\end{proposition}

\subsection{Lower bounds from slices}

The spectral norm admits two useful characterizations in terms of slices. Let again $\bX_{i_\mu} = \bX(\colon,\dots,\colon,i_\mu,\colon,\dots,\colon) \in \V^{[\mu]}$ denote the slices of a tensor $\bX$ perpendicular to mode $\mu$. The following formula is immediate from~\eqref{eq: recursive contraction} and the commutativity of partial contractions:
\[
\| \bX \|_2 = \max_{\substack{u^\mu \in \K^{n_\mu} \\ \| u^\mu \|_{V^\mu} = 1 }}  \| \bX \times_\mu u^\mu \|_2
 = \max_{\substack{u^\mu \in \K^{n_\mu} \\ \| u^\mu \|_{V^\mu} = 1 }} \left\| \sum_{i_\mu = 1}^{n_\mu} u^\mu(i_\mu) \bX_{i_\mu} \right\|_2.
\]
By choosing $u^\mu=e_i$ (the $i$th column of the identity matrix), we conclude that
\begin{equation}\label{eq: lower bound from slices}
\| \bX_{i_\mu} \|_2 \le \| \bX \|_2
\end{equation}
for all slices.

We also have the following.
\begin{proposition}\label{prop: characterization of spectral norm}
\[
\| \bX \|_2 = \max_{\substack{\bZ \in \mC_1^{[\mu]} \\ \| \bZ \|_\frob = 1}} \left( \sum_{i_\mu = 1}^{n_\mu} \abs{\langle \bX_{i_\mu}, \bZ \rangle_\frob}^2 \right)^{1/2}.
\]
\end{proposition}
\begin{proof}
Since the spectral norm is invariant under permutation of indices, it is enough to show this for $\mu=1$. We can write
\[
\| \bX \|_2 =  \max_{\substack{\bZ \in \mC_1^{[1]} \\ \| \bZ \|_\frob = 1}} \max_{\| u^1 \|_{V^1} = 1}
\langle \bX, u^1 \otimes \bZ \rangle_\frob = \max_{\substack{\bZ \in \mC_1^{[1]} \\ \| \bZ \|_\frob = 1}} \max_{\| u^1 \|_{V^1} = 1} \sum_{i_1 = 1}^{n_1} \langle \bX_{i_1}, \bZ \rangle_\frob \cdot u^1(i_1).
\]
By the Cauchy--Schwarz inequality, the inner maximum is achieved for $u^1 = x / \| x \|$ with $x(i_1) = \overline{\langle \bX_{i_1}, \bZ \rangle_\frob}$ for $i_1=1,\dots,n_1$. This yields the assertion.
\end{proof}

\subsection{Upper bounds from matricizations}

Let $t \subsetneq \{1,\dots,d\}$ be nonempty. Then there exists a natural isometric isomorphism between the spaces $\K^{n_1 \times \dots \times n_d}$ and $\K^{\prod_{\mu \in t} n_\mu} \otimes \K^{\prod_{\nu \notin t}n_\nu}$. This isomorphism is called $t$-matricization (or $t$-flattening). More concretely, we can define two multi-index sets
\[
\mathbf{I}^t = \bigtimes_{\mu \in t} \{1,\dots,n_\mu\}, \quad \mathbf{J}^t = \bigtimes_{\nu \notin t} \{1,\dots,n_\nu\}.
\]
Then a tensor $\bX$ yields, in an obvious way, an $(\prod_{\mu \in t} n_\mu) \times (\prod_{\nu \notin t} n_\nu)$ matrix $\bX^t$ with entries
\begin{equation}\label{eq: t-matricization}
\bX^t(\mathbf{i},\mathbf{j}) = \bX(i_1,\dots,i_d), \quad \mathbf{i} \in \mathbf{I}^t, \ \mathbf{j} \in \mathbf{J}^t.
\end{equation}

The main observation is that $\bX^t$ is a rank-one matrix if $\bX$ is an elementary tensor (the converse is not true in general). Since we can always construct a tensor from its $t$-matricization, we obtain from the best-rank one approximation ratio for matrices that
\[
\App_d(\K;n_1,\dots,n_d) \le \frac{1}{\sqrt{\min\left( \prod_{\mu \in t} n_\mu, \prod_{\nu \notin t} n_\nu \right)}	}.
\]
This is because $\|\bX\|_2 \le \|\bX^t\|_2$ and $\|\bX\|_\frob = \|\bX^t\|_\frob$. Here the subset $t$ is arbitrary. In combination with~\eqref{eq: trivial lower bound}, this allows the following conclusion for tensors with one dominating mode size.

\begin{proposition}\label{prop: sharpness for tall tensors}
If there exists $\nu \in \{1,\dots,d\}$ such that $\prod_{\mu \neq \nu} n_\mu \le n_\nu$, then
\[
\App_d(\K;n_1,\dots,n_d) = \frac{1}{\sqrt{\min_{\nu = 1,\dots,d} \prod_{\mu \neq \nu} n_\mu}},
\]
that is, the trivial bound~\eqref{eq: trivial lower bound} is sharp.
\end{proposition}

For instance, $\App_3(\K;n,n,n^2) = 1/n$.

\subsection{Upper bounds from random tensors}\label{sec: random results}

We conclude this section with some known upper bounds derived from considering random tensors. These results are obtained by combining coverings of the set of normalized (to Frobenius norm one) elementary tensors with concentration of measure results.

In~\cite{GrossFlamiaEisert2009}, Gross, Flammia, and Eisert showed that for $d\ge 11$ the fraction of tensors $\bX$ on the unit sphere in $\C^{2 \times \dots \times 2}$ satisfying
\[
\| \bX \|_2^2 \le \frac{1}{2^{d - 2 \log_2 d - 3}}
\]
is at least $1-e^{-d^2}$.

More recently, Tomioka and Suzuki~\cite{Tomioka2014} provided a simplified version of a result by Nguyen, Drineas, and Tran~\cite{NguyenDrineasTran2015}, namely that
\[
\| \bX \|_2^2 \le C \ln d \sum_{\mu = 1}^d n_\mu
\]
with any desired probability for real tensors with independent, zero-mean, sub-Gaussian entries satisfying $\mathbb{E}(\mathrm{e}^{t\bX{i_1,\ldots,i_d}})\leq \mathrm{e}^{\sigma^2t^2/2}$, as long as the constant $C$ is taken large enough.
For example, when the elements are independent and identically distributed Gaussian, we have
\begin{equation*}
    \| \bX \|_2^2 \leq C \ln d \sum_{\mu=1}^d n_\mu, \quad
    \| \bX \|_\frob^2 \geq C' n_1 \cdots n_d
\end{equation*}
with probability larger than $1/2$, respectively, where the second inequality follows from the tail bound of the $\chi$-squared distribution. Thus,
\begin{equation}
  \label{eq:randombound}
    \frac{\| \bX \|_2^2}{ \| \bX \|_\frob^2} \le \frac{C'' \ln d \sum_{\mu=1}^d n_\mu}{ n_1 \cdots n_d } \le
\frac{ C'' d \ln d}{\min_{\nu = 1,\dots,d} \prod_{\mu \neq \nu} n_\mu}
\end{equation}
with positive probability. This shows that the naive lower bound~\eqref{eq: trivial lower bound}, whether sharp or not, provides the right order of magnitude for $\App(\V)$ (at least when $\mathbb{K}=\mathbb{R}$).

For cubic tensors this was known earlier. By inspecting the expectation of spectral norm of random $n \times n \times n$ tensors, Cobos, K\"uhn, and Peetre~\cite{CobosKuehnPeetre1999} obtained the remarkable estimates
\begin{equation}\label{eq: asymptotic behavior nxnxn}
\frac{1}{n} \le \App_3(\R;n,n,n) \le \frac{3\sqrt{\pi}}{\sqrt{2}}\frac{1}{n}
\end{equation}
and
\[
  \frac{1}{n} \le \App_3(\C;n,n,n) \le {3\sqrt{\pi}}\frac{1}{n}.
\]
They also remark, without explicit proof, that $\App_d(\K;n,\dots,n) = O(1/\sqrt{n^{d-1}})$, in particular
\[
 \App_d(\R;n,\dots,n) \le \frac{d \sqrt{\pi}}{\sqrt{2}} \frac{1}{\sqrt{n^{d-1}}}.
\]
Note that the estimate~\eqref{eq:randombound} provides a slightly better scaling of $\App_d(\R;n\dots,n)$ with respect to $d$, namely, $\sqrt{d \ln d}$ instead of $d$.

\section{Orthogonal and unitary tensors}\label{sec: orthogonal tensors}

In this section we introduce the concept of orthogonal tensors. It is a ``natural'' extension of matrices with pairwise orthonormal rows or orthonormal columns. Orthogonal matrices play a fundamental role in both matrix analysis~\cite{HornJohnson1985} and numerical computation~\cite{golubbook4th,parlettsym}. Although the concept of orthogonal tensors was proposed earlier in~\cite{GnangElgammalRetakh2011}, we believe that our less abstract definition given below extends naturally from some properties of matrices with orthonormal rows or columns. As in the matrix case, we will see in the next section that orthogonality is a necessary and sufficient condition for a tensor to achieve the trivial bound~\eqref{eq: trivial lower bound} on the extreme ratio between spectral and Frobenius norms. However, it also turns out that orthogonality for tensors is a very strong property and in many tensor spaces (configurations of $(n_1,\dots,n_d)$ and the field $\K$) orthogonal tensors do not exist.

For ease of presentation we assume in the following that $n_1 \le \dots \le n_d$, but all definitions and results transfer to general tensors using suitable permutations of dimensions. In this sense, our recursive definition of orthogonal tensors generalizes matrices with pairwise orthonormal rows.

\begin{definition} \label{def:ot1}
A tensor of order one, i.e., a vector $u^1 \in \K^{n_1}$, is called \emph{orthogonal} for $\K = \R$ (resp., \emph{unitary} for $\K = \C$) if its Euclidean norm equals one (unit vector). Let $n_1 \le n_2 \le \dots \le n_d$. Then $\bX \in \K^{n_1 \times \dots \times n_d}$ is called \emph{orthogonal} for $\K = \R$ (resp., \emph{unitary} for $\K = \C$), if for every unit vector $u^1 \in \K^{n_1}$, the tensor $\bX \times_1 u^1$ is orthogonal (resp., unitary).
\end{definition}

Since partial contractions commute, one could use the following, slightly more general definition of orthogonal (unitary) tensors of order $d \ge 2$ (which, e.g., subsumes matrices with orthonormal rows or columns). Let $\nu$ be such that $n_{\nu} \ge n_\mu$ for all $\mu$. Then $\bX$ is orthogonal (unitary) if for any subset $S\subset\{1,2,\dots,d\} \setminus \{ \nu\}$ and any unit vectors $u^\mu\in\K^{n_\mu}$, the tensor $\bX\times_{\mu\in S} u^\mu$ of order $d-|S|$ is orthogonal (unitary). In particular, $\bX\times_\mu u^\mu$ is an orthogonal (unitary) tensor of order $d-1$ for any $\mu \neq \nu$. It is clear that $\bX$ will be orthogonal (unitary) according to this definition if and only if for any permutation $\pi$ of $\{1,\dots,d\}$ the tensor with entries $\bX(i_{\pi(1)},\dots,i_{\pi(d)})$ is orthogonal (unitary). Therefore, we can stick without loss of generality to consider the case where $n_1 \le \dots \le n_d$ and the Definition~\ref{def:ot1} of orthogonality (unitarity).

An alternative way to think of orthogonal and unitary tensors is as length-preserving $(d-1)$-form in the following sense. Every tensor $\bX \in \K^{n_1 \times \dots \times n_d}$ defines a $(d-1)$-linear form
\begin{equation}\label{eq:induced d-1 form}
\begin{gathered}
    \omega_\bX \colon {\K}^{n_1} \times \dots \times {\K}^{n_{d-1}} \to {\K}^{n_d}, \\
 (u^1,\dots,u^{d-1}) \mapsto \bX \times_1 u^1 \dots \times_{d-1} u^{d-1}.
\end{gathered}
\end{equation}
It is easy to obtain the following alternative, noninductive definition of orthogonal (unitary) tensors.

\begin{proposition}\label{prop:length-preserving}
Let $n_1 \le \dots \le n_d$. Then $\bX \in \K^{n_1 \times \dots \times n_d}$ is orthogonal (unitary) if and only if
\[
\| \omega_\bX(u^1,\dots,u^{d-1}) \|_2 = \prod_{\mu=1}^{d-1} \| u^\mu \|_2
\]
for all $u^1,\dots,u^{d-1}$.
\end{proposition}

For third-order tensors this property establishes an equivalence between orthogonal tensors and the Hurwitz problem that will be discussed in section~\ref{sec: Hurwitz problem}. By considering subvectors of $u^1,\dots,u^{d-1}$, it further proves the following fact.

\begin{proposition}\label{prop: orthogonality of subtensors}
 Let $n_1 \le \dots \le n_d$ and $\bX \in \K^{n_1 \times \dots \times n_d}$ be orthogonal (unitary). Then any $n_1' \times \dots \times n_{d-1}' \times n_d$ subtensor of $\bX$ is also orthogonal (unitary).
\end{proposition}

We now list some extremal properties of orthogonal and unitary tensors related to the spectral norm, nuclear norm and orthogonal rank.

\begin{proposition}\label{prop: properties of orthogonal tensors}
 Let $n_1 \le \dots \le n_d$ and $\bX \in \K^{n_1 \times \dots \times n_d}$ be orthogonal or unitary. Then
 \begin{itemize}
  \item[\upshape (a)]
  $\displaystyle \| \bX \|_2 = 1, \quad \| \bX \|_\frob = \sqrt{\prod_{\mu = 1}^{d-1} n_\mu}, \quad \| \bX \|_* = \prod_{\mu = 1}^{d-1} n_\mu$,
  \item[\upshape (b)]
  $\rank_\bot (\bX)  = \displaystyle \prod_{\mu = 1}^{d-1} n_\mu$.
 \end{itemize}
\end{proposition}

\begin{proof}
 Ad (a). It follows from orthogonality that all fibers $\bX(i_1,\dots,i_{d-1},\colon)$ along dimension $n_d$ have norm one (because the fibers can be obtained from contractions with standard unit vectors). There are $\prod_{\mu=1}^{d-1} n_\mu$ of such fibers, hence $\| \bX \|_\frob^2 = \prod_{\mu=1}^{d-1} n_\mu$. From the trivial bound~\eqref{eq: trivial lower bound} it then follows $\| \bX \|_2 \ge 1$. On the other hand, by the Cauchy--Schwarz inequality and orthogonality (Proposition~\ref{prop:length-preserving}),
\begin{align*}
\abs{ \langle \bX, u^1 \otimes \dots \otimes u^d \rangle_\frob} &= \abs{ \langle \omega_\bX(u^1, \dots , u^{d-1}), u^d \rangle_{\K^{n_d}}} \\ &\le \| \omega_\bX ( u^1, \dots u^{d-1}) \|_2 \| u^d \|_2 \le \prod_{\mu=1}^d \| u^\mu \|_2.
\end{align*}
Hence $\| \bX \|_2 \le 1$. Now~\eqref{eq: duality estimate} and~\eqref{eq: trivial bound for nuclear norm} together give the asserted value of $\| \bX \|_*$.

Ad (b). Due to (a), this follows by combining~\eqref{eq: estimating by orthogonal rank} and~\eqref{eq: trivial bound for orthogonal rank}.
\end{proof}

Our main aim in this section is to establish that, as in the matrix case, the extremal values of the spectral and nuclear norms in Proposition~\ref{prop: properties of orthogonal tensors} fully characterize multiples of orthogonal and unitary tensors.

\begin{theorem}\label{th: main theorem for orthogonal tensors}
Let $n_1 \le \dots \le n_d$ and $\bX \in \K^{n_1 \times \dots \times n_d}$, $\bX \neq 0$. The following are equivalent:
\begin{itemize}
\item[\upshape (a)]
$\bX$ is a scalar multiple of an orthogonal (resp., unitary) tensor,
\item[\upshape (b)]
$\displaystyle \frac{\| \bX \|_2}{\| \bX \|_\frob} = \frac{1}{\sqrt{\prod_{\mu =1}^{d-1} n_\mu}}$,
\item[\upshape (c)]
$\displaystyle \frac{\| \bX \|_*}{\| \bX \|_\frob} = \sqrt{\prod_{\mu =1}^{d-1} n_\mu}$.
\end{itemize}
\end{theorem}

In light of the trivial lower bound~\eqref{eq: trivial lower bound} on the spectral norm, and the relation~\eqref{eq: equivalence for spectral and nuclear} with the nuclear norm, the immediate conclusion from this theorem is the following.

\begin{corollary}\label{cor: main corollary}
Let $n_1 \le \dots \le n_d$. Then
\[
\App_d(\K; n_1, \dots,n_d) = \frac{1}{\sqrt{\prod_{\mu =1}^{d-1} n_\mu}}
\]
if and only if orthogonal (resp. unitary) tensors exist in $\K^{n_1 \times \dots \times n_d}$. Otherwise, the value of $\App_d(\K; n_1, \dots,n_d)$ is strictly larger. Analogously, it holds that
\[
 \max_{\bX \neq 0} \frac{ \| \bX \|_*}{ \| \bX \|_\frob} = \sqrt{\prod_{\mu =1}^{d-1} n_\mu}
\]
in $\K^{n_1 \times \dots \times n_d}$ if and only if orthogonal (resp. unitary) tensors exist.
\end{corollary}

\begin{proof}[Proof of Theorem~\ref{th: main theorem for orthogonal tensors}]
In the proof, we use the notation $n_1 \cdots n_{d-1}$ instead of $\prod_{\mu=1}^{d-1} n_\mu$. By Proposition~\ref{prop: properties of orthogonal tensors}, (a) implies (b) and (c).

\smallskip

We show that (b) implies (a). The proof is by induction over $d$. For $d=1$ the spectral norm and Frobenius norm are equal. When $d=2$, we have already mentioned in section~\ref{sec: matrices} that for $m \le n$ only $m \times n$ matrices with pairwise orthonormal rows achieve $\| \bX \|_\frob = \sqrt{m}$ and $\| \bX \|_2 = 1$. Let now $d \ge 3$ and assume (b) always implies (a) for tensors of order $d-1$. Consider $\bX \in \K^{n_1 \times \dots \times n_d}$ with $\| \bX \|_\frob^2 = n_1 \cdots n_{d-1}$ and $\| \bX \|_2 = 1$. Then all the $n_1 \cdots n_{d-1}$ fibers $\bX(i_1,\dots,i_{d-1},\colon)$ parallel to the last dimension have Euclidean norm one, since otherwise one of these fibers has a larger norm, and so the corresponding rank-one tensor containing only that fiber (but normalized) provides a larger overlap with $\bX$ than one. As a consequence, the $n_1$ slices $\bX_{i_1} = \bX(i_1,\colon,\dots,\colon) \in \K^{n_2 \times \dots \times n_{d}}$, $i_1 = 1,\dots,n_1$, have squared Frobenius norm $n_2 \cdots n_d$ and spectral norm one (by~\eqref{eq: trivial lower bound}, $\| \bX_{i_1} \|_2 \ge 1$, whereas by~\eqref{eq: lower bound from slices}, $\| \bX_{i_1} \|_2 \le 1$). It now follows from the induction hypothesis and Proposition~\ref{prop: properties of orthogonal tensors} that all slices are orthogonal (resp., unitary) tensors.

Now let $u^1 \in \K^{n_1}$, \dots, $u^{d-1} \in \K^{n_{d-1}}$ have norm one. We have to show that
\begin{align*}
\omega_\bX(u^1,\dots,u^{d-1}) &= \bX \times_1 u^1 \times_2 u^2 \dots \times_{d-1} u^{d-1} \\
& = \sum_{i_1 = 1}^{n_1} u^1(i_1) \left( \bX_{i_1} \times_2 u^2 \dots \times_{d-1} u^{d-1} \right)
\end{align*}
has norm one.\footnote{The notation $\bX_{i_1} \times_2 u^2 \dots \times_{d-1} u^{d-1}$ is convenient although slightly abusive, since, e.g., $\times_2$ is strictly speaking a contraction in the first mode of $\bX_{i_1}$.} Since the $\bX_{i_1}$ are orthogonal (resp., unitary), the vectors $v^{i_1} = \bX_{i_1} \times_2 u^2 \dots \times_{d-1} u^{d-1}$ have norm one. It is enough to show that they are pairwise orthogonal in $\K^{n_d}$. Without loss of generality assume to the contrary that $\langle v^1, v^2 \rangle \neq 0$. Then the matrix $M\in\K^{2\times n_d}$ with rows $v^1$ and $v^2$
 has spectral norm larger than one. Hence there exist $\tilde{u} \in \K^2$ and $u^d \in \K^{n_d}$, both of norm one, such that for $u^1 = (\tilde u (1), \tilde u(2),0,\dots,0) \in \K^{n_1}$ it holds that
\[
\abs{\langle \bX, u^1 \otimes u^2 \otimes \dots \otimes u^{d-1} \otimes u^d \rangle_\frob} = \abs{\tilde u(1) \langle v^1, u^d \rangle + \tilde u(2) \langle v^2, u^d \rangle} = \abs{{\tilde u}^T M u^d} > 1.
\]
This contradicts $\| \bX \|_2 = 1$.

\smallskip

We prove that (c) implies (b). Strictly speaking, this follows from~\cite[Thm.~2.2]{DerksenFriedlandLimWang2017}, which states that $\| \bX \|_* / \| \bX \|_\frob = (\App(\V))^{-1}$ if and only if $\| \bX \|_2 / \| \bX \|_\frob = \App(\V)$, and (c) implies the first of these properties (by~\eqref{eq: trivial bound for nuclear norm} and~\eqref{eq: equivalence for spectral and nuclear}). The following more direct proof is still insightful.

If (c) holds, we can assume that
\begin{equation}\label{scaling assumption}
\| \bX \|_* = n_1 \cdots n_{d-1} = \| \bX \|_\frob^2.
\end{equation}
By Proposition~\ref{prop: decomposition into normal form}, we can find a decomposition $\bX = \sum_{k=1}^{n_1 \cdots n_{d-1}} \bZ_k$ into $n_1 \cdots n_{d-1}$ mutually orthogonal elementary tensors $\bZ_k \in \mC_1$ such that $\| \bZ_1 \|_\frob = \| \bX \|_2$.  Using the definition~\eqref{eq: definition of nuclear norm} of nuclear norm, the Cauchy--Schwarz inequality, and~\eqref{scaling assumption} we obtain
\[
\| \bX \|_* \le \sum_{k=1}^{n_1 \cdots n_{d-1}} \| \bZ_k \|_\frob \le \sqrt{n_1 \cdots n_{d-1}} \| \bX \|_\frob = \| \bX \|_*.
\]
Hence the inequality signs are actually equalities. However, equality in the Cauchy--Schwarz inequality is attained only if all $\| \bZ_k \|_\frob$'s take the same value, namely,
\[
 \| \bZ_k \|_\frob = \frac{ \| \bX \|_\frob}{\sqrt{n_1 \cdots n_{d-1}}} = 1.
\]
In particular, $\| \bZ_1 \|_\frob = \| \bX \|_2$ has this value, which shows (b).
\end{proof}

\begin{remark}\label{rem:nonuniqueness}
We note for completeness that by Proposition~\ref{prop: properties of orthogonal tensors} an orthogonal (resp., unitary) tensor has infinitely many best rank-one approximations and they are very easy to construct. In fact, given {\em any} unit vectors $u^\mu \in \K^{n_\mu}$ for $\mu=1,\dots,d-1$, let $u^d = \bX \times_1 u^1 \dots \times_{d-1} u^{d-1}$, which is also a unit vector. Then
\[
\abs{\langle \bX, u^1 \otimes \dots \otimes u^d \rangle_\frob} = \| u^d \|_2^2 = 1 = \| \bX \|_2,
\]
which, by Proposition~\ref{prop}, shows that $u^1 \otimes \dots \otimes u^d$ is a best rank-one approximation of $\bX$.
\end{remark}

\section{Existence of orthogonal and unitary tensors}\label{sec: existence}

\subsection{Third-order tensors}

For a third-order tensor $\bX \in \K^{\ell \times m \times n}$ with $\ell \le m \le n$, the lower bound~\eqref{eq: trivial lower bound} takes the form
\begin{equation}\label{eq: lower bound for third order}
\frac{\| \bX \|_2}{\| \bX \|_\frob} \ge \frac{1}{\sqrt{\ell m}}.
\end{equation}
By Theorem~\ref{th: main theorem for orthogonal tensors}, equality can be achieved only for orthogonal (resp., unitary) tensors. From Proposition~\ref{prop: sharpness for tall tensors} we know that this estimate is sharp in the case $\ell m \le n$. In fact, an orthogonal tensor can then be easily constructed via its slices
\[
\bX(i,\colon,\colon) = [
\underbrace{
  \begin{matrix}
  O  & \cdots & O
  \end{matrix}}_{i-1}
\begin{matrix}
& Q_i & O & \cdots & O
\end{matrix}
 ] \in  \K^{m \times n}, \quad i=1,\dots,\ell,
\]
where the entries represent blocks of size $m \times m$ (except the last block might have fewer or even no columns), and the $Q_i\in\K^{m\times m}$ are matrices with pairwise orthonormal rows at position $i$.

In this section we inspect the sharpness in the case $\ell m > n$, where such a construction is not possible in general. Interestingly, the results depend on the underlying field.

\subsubsection{Real case: Relation to Hurwitz problem}\label{sec: Hurwitz problem}

By Proposition~\ref{prop:length-preserving}, a third-order tensor $\bX \in \K^{\ell \times m \times n}$ is orthogonal if and only if the bilinear form $\omega_\bX(u,v) = \bX \times_1 u \times_2 v$ satisfies
\begin{equation}\label{eq: length preserving}
\| \omega_\bX(u,v) \|_2 = \| u \|_2 \| v \|_2
\end{equation}
for all $u \in \R^\ell$, $v \in \R^m$. In the real case $\K = \R$, this relation can be written as
\begin{equation}\label{eq:quadratesatze}
\sum_{k = 1}^n \omega_k(u,v)^2 =  \left( \sum_{i = 1}^\ell u_i^2 \right) \left( \sum_{j=1}^m v_j^2 \right).
\end{equation}
The question of whether for a given triple $[\ell,m,n]$ of dimensions a bilinear form $\omega(u,v)$ exists obeying this relation is known as the Hurwitz problem (here for the field $\R$). If a solution exists, the triple $[\ell,m,n]$ is called \emph{admissible} for the Hurwitz problem. \rr{Since, on the other hand, the correspondence $\bX \mapsto \omega_\bX$ is a bijection\footnote{The inverse is given through $\bX(i,j,k) = \omega_k(e_i,e_j)$ with standard unit vectors $e_i$, $e_j$.} between $\R^{\ell \times m \times n}$ and the space of bilinear forms $\R^\ell \times \R^m \to \R^n$, every solution to the Hurwitz problem yields an orthogonal tensor.} \rr{For real third-order tensors, Theorem~\ref{th: main theorem for orthogonal tensors} can hence be stated as follows.}

\begin{theorem}\label{th: admissibility}
Let $\ell \le m \le n$. A tensor $\bX \in \R^{\ell \times m \times n}$ is orthogonal if and only if the induced bilinear form $\omega_\bX$ is a solution to the Hurwitz problem~\eqref{eq:quadratesatze}. Correspondingly, it holds that
\[
\App_3(\R;\ell,m,n) = \frac{1}{\sqrt{\ell m}}
\]
if and only if $[\ell,m,n]$ is an admissible triple for the Hurwitz problem.
\end{theorem}

Some admissible cases (besides $\ell m \le n$) known from the literature are discussed next.

\subsubsection*{$n\times n \times n$ tensors and composition algebras}

In the classical work~\cite{Hurwitz1898}, Hurwitz considered the case $\ell = m = n$. In this case the bilinear form $w_\bX$ turns $\R^n$ into an algebra on $\R^n$. In modern terminology, an algebra on $\R^n$ satisfying the relation~\eqref{eq:quadratesatze} for its product $u \cdot v = \omega$ is called a composition algebra. Hurwitz disproved the existence of such an algebra for the cases $n \neq 1,2,4,8$.\footnote{In fact, when $\bX$ is orthogonal, $\omega_\bX$ turns $\R^n$ into a division algebra. By a much deeper result, these algebras also only exist for $n = 1,2,4,8$.}

For the cases $n=1,2,4,8$, the real field $\R$, the complex field $\C$, the quaternion algebra $\mathbb{H}$, and the octonion algebra $\mathbb{O}$ are composition algebras on $\R^n$, respectively, since the corresponding multiplications are length preserving. Consequently, examples for orthogonal $n \times n \times n$ tensors are given by the multiplication tensors of these algebras. For completeness we list them here.

For $n = 1$ this is just $\bX = 1$. For $n = 2$, let $e_1,e_2$ denote the standard unit vectors in $\R^2$, i.e., $[e_1\ e_2]=I_2$. Then
\begin{equation}\label{eq:complex mult}
\bX_{\C} = \begin{bmatrix*}[r] e_1 & e_2 \\ e_2 & - e_1 \end{bmatrix*} \in \R^{2 \times 2 \times 2}
\end{equation}
is orthogonal. This is the tensor of multiplication in $\C \cong \R^2$. \rr{Here (and in the following), the matrix notation with vector-valued entries means that $\bX_{\C}$ has the fibers $\bX_{\C}(1,1,:) = e_1$, $\bX_{\C}(1,2,:) = e_2$, $\bX_{\C}(2,1,:) = e_2$, and $\bX_{\C}(2,2,:) = -e_1$ along the third mode.}

For $n=4$, let $e_1,e_2,e_3,e_4$ denote the standard unit vectors in $\R^4$; then
\begin{equation}\label{eq:quaternion mult}
\bX_{\mathbb{H}} = \begin{bmatrix*}[r]
e_1 & e_2 & e_3 & e_4 \\
e_2 & -e_1 & e_4 & -e_3 \\
e_3 & -e_4 & -e_1 & e_2 \\
e_4 & e_3 & -e_2 & -e_1
\end{bmatrix*} \in \R^{4 \times 4 \times 4}
\end{equation}
is orthogonal. This is the tensor of multiplication in the quaternion algebra $\mathbb{H} \cong \R^4$.

For $n = 8$, let $e_1,\dots,e_8$ denote the standard unit vectors in $\R^8$; then
\begin{equation}\label{eq:octonion mult}
\bX_{\mathbb{O}} = \begin{bmatrix*}[r]
e_1 & e_2 & e_3 & e_4 & e_5 & e_6 & e_7 & e_8 \\
e_2 & -e_1 & e_4 & -e_3 & e_6 & -e_5 & -e_8 & e_7 \\
e_3 & -e_4 & -e_1 & e_2 & e_7 & e_8 & -e_5 & -e_6 \\
e_4 & e_3 & -e_2 & -e_1 & e_8 & -e_7 & e_6 & -e_5 \\
e_5 & -e_6 & -e_7 & -e_8 & -e_1 & e_2 & e_3 & e_4 \\
e_6 & e_5 & -e_8 & e_7 & -e_2 & -e_1 & -e_4 & e_3 \\
e_7 & e_8 & e_5 & -e_6 & -e_3 & e_4 & -e_1 & -e_2 \\
e_8 & -e_7 & e_6 & e_5 & -e_4 & -e_3 & e_2 & -e_1
\end{bmatrix*} \in \R^{8 \times 8 \times 8}
\end{equation}
is orthogonal. This is the tensor of multiplication in the octonion algebra $\mathbb{O} \cong \R^8$.

For reference we summarize the $n \times n \times n$ case.

\begin{theorem}\label{th: Hurwitz result}
Real orthogonal $n \times n \times n$ tensors exist only for $n = 1,2,4,8$.
Consequently,
\[
\App_3(\R;n,n,n) = \frac{1}{n}
\]
if and only if $n = 1,2,4,8$. Otherwise, the value of $\App_3(\R;n,n,n)$ must be strictly larger.
\end{theorem}

\subsubsection*{Other admissible triples}

There exists an impressive body of work for identifying admissible triples for the Hurwitz problem. The problem can be considered as open in general. We list some of the available results here.  We refer to~\cite{Shapiro2000} for an introduction into the subject and to~\cite{Lenzhenetal2011} for recent results and references.

Regarding triples $[\ell,m,n]$ with $\ell \le m \le n$ we can observe that if a configuration is admissible, then so is $[\ell',m',n']$ with $\ell' \le \ell$, $m' \le m$, and $n' \ge n$. This follows directly from~\eqref{eq:quadratesatze}, since we can consider subvectors of $u$ and $v$ and artificially expand the left sum with $\omega_k = 0$. As stated previously, $n \ge \ell m$ is always admissible. Let
\[
 \ell * m \coloneqq \min\{n \colon \text{$[\ell,m,n]$ is admissible}  \},
\]
i.e., the minimal $n$ for~\eqref{eq:quadratesatze} to exist. For $\ell \le 9$ these values can be recursively computed for all $m \ge \ell$ according to the rule~\cite[Prop.~12.9 and 12.13]{Shapiro2000}:
\[
 \ell * m = \begin{cases} 2(\lceil \frac{\ell}{2} \rceil * \lceil \frac{m}{2} \rceil) - 1 &\quad \!\!\!\! \text{if $\ell$, $m$ are both odd and $\lceil \frac{\ell}{2} \rceil * \lceil \frac{m}{2} \rceil = \lceil \frac{\ell}{2} \rceil + \lceil \frac{m}{2} \rceil - 1$,}\\
  2(\lceil \frac{\ell}{2} \rceil * \lceil \frac{m}{2} \rceil) &\quad \!\!\!\! \text{else.}
 \end{cases}
\]
This provides the following table~\cite{Shapiro2000}:
\begin{equation*}\label{table for 1-9}
\begin{tabular}{c|ccccccccccccccccc}
$\ell \setminus m$ & 1 & 2 & 3 & 4 & 5 & 6 & 7 & 8 & 9 & 10 & 11 & 12 & 13 & 14 & 15 & 16 \\\hline
1 & 1 & 2 & 3 & 4 & 5 & 6 & 7 & 8 & 9 & 10 & 11 & 12 & 13 & 14 & 15 & 16 \\
2 &   & 2 & 4 & 4 & 6 & 6 & 8 & 8 & 10 & 10 & 12 & 12 & 14 & 14 & 16 & 16 \\
3 &   &   & 4 & 4 & 7 & 8 & 8 & 8 & 11 & 12 & 12 & 12 & 15 & 16 & 16 & 16 \\
4 &   &   &   & 4 & 8 & 8 & 8 & 8 & 12 & 12 & 12 & 12 & 16 & 16 & 16 & 16 \\
5 &   &   &   &   & 8 & 8 & 8 & 8 & 13 & 14 & 15 & 16 & 16 & 16 & 16 & 16 \\
6 &   &   &   &   &   & 8 & 8 & 8 & 14 & 14 & 16 & 16 & 16 & 16 & 16 & 16 \\
7 &   &   &   &   &   &   & 8 & 8 & 15 & 16 & 16 & 16 & 16 & 16 & 16 & 16 \\
8 &   &   &   &   &   &   &   & 8 & 16 & 16 & 16 & 16 & 16 & 16 & 16 & 16 \\
9 &   &   &   &   &   &   &   &   & 16 & 16 & 16 & 16 & 16 & 16 & 16 & 16
\end{tabular}
\end{equation*}
For $10 \le \ell \le 16$, the following table due to Yiu~\cite{Yiu1994} provides upper bounds for $\ell * m$ (in particular it yields admissible triples):
\begin{equation}\label{table for 10-16}
\begin{tabular}{c|ccccccc}
$\ell \setminus m$ & 10 & 11 & 12 & 13 & 14 & 15 & 16 \\\hline
10 & 16 & 26 & 26 & 27 & 27 & 28 & 28 \\
11 &    & 26 & 26 & 28 & 28 & 30 & 30 \\
12 &    &    & 26 & 28 & 30 & 32 & 32 \\
13 &    &    &    & 28 & 32 & 32 & 32 \\
14 &    &    &    &    & 32 & 32 & 32 \\
15 &    &    &    &    &    & 32 & 32 \\
16 &    &    &    &    &    &    & 32
\end{tabular}
\end{equation}

\rr{The admissible triples in these tables are obtained by rather intricate combinatorial constructions of solutions $\omega = \omega_\bX$ to the Hurwitz problem~\eqref{eq:quadratesatze}, whose tensor representations $\bX$ have integer entries (integer composition formulas); see~\cite[p.~269~ff.]{Shapiro2000} for details. From the abstract construction, it is not easy to directly write down the corresponding orthogonal tensors, although in principle it is possible. For the values in the table~\eqref{table for 10-16} it is not known whether they are smallest possible if one admits real entries in $\bX$ as we do here (although this is conjectured~\cite[p.~314]{Shapiro2000}). Some further upper bounds for $\ell * m$ based on integer composition formulas for larger values of $\ell$ and $m$ are listed in~\cite[p.~291~ff.]{Shapiro2000}.}

There are also nontrivial infinite families of admissible triples known. Radon~\cite{Radon1922} and Hurwitz~\cite{Hurwitz1922} independently determined the largest $\ell \le n$ for which the triple $[\ell,n,n]$ is admissible: writing $n = 2^{4\alpha + \beta} \gamma$ with $\beta \in \{0,1,2,3\}$ and $\gamma$ odd, the maximal admissible value of $\ell$ is
\begin{equation}\label{eq: Radon-Hurwitz}
\ell_{\text{max}} = 2^\beta + 8 \alpha.
\end{equation}
\rr{If $n \ge 2$ is even, then $\ell_{\text{max}} \ge 2$, and so
 \[
  \App_3(\R;\ell,n,n) = \frac{1}{\sqrt{\ell n}} \quad \text{for $n$ even and $1 \le \ell \le \ell_{\text{max}}$.}
 \]
In particular, we recover~\eqref{eq: kongetal} as a special case. On the other hand, when $n$ is odd, then $\alpha = \beta = 0$ and so $\ell_{\text{max}} = 1$. Hence $[\ell,n,n]$ is not admissible for $\ell \ge 2$ and $\App_3(\R;\ell,n,n) > 1/\sqrt{\ell n}$, in line, e.g., with~\eqref{eq: kongetal results}.
}

Some known families of admissible triples ``close'' to Hurwitz--Radon triples are
\[
\left[2 + 8\alpha, 2^{4\alpha} - \binom{4\alpha}{2\alpha},2^{4\alpha} \right]
\quad
\text{and}
\quad
[2\alpha, 2^\alpha - 2\alpha, 2^\alpha - 2],
\quad \text{$\alpha \in \mathbb{N}$.}
\]
We refer once again to~\cite{Lenzhenetal2011} for more results of this type.

\subsubsection{Complex case}

In the complex case, the answer to the existence of unitary tensors in the case $\ell m > n$ is very simple: they do not exist. For example, for complex $2 \times 2 \times 2$ tensors this is illustrated by the fact that $\App_3(\C;2,2,2) = 2/3$; see~\cite{DerksenFriedlandLimWang2017}.

\begin{theorem}\label{th: nonexistence in complex case}
Let $\ell \le m \le n$. When $\ell m > n$, there exists no unitary tensor in $\C^{\ell \times m \times n}$, and hence
\[
\App_3(\C;\ell,m,n) > \frac{1}{\sqrt{\ell m}}.
\]
\end{theorem}

\begin{proof}
Suppose to the contrary that some $\bX \in \C^{\ell \times m \times n}$ is unitary. Let $X_i = \bX(i,\colon,\colon) \in \C^{m \times n}$ denote the slices of $\bX$ perpendicular to the first mode. By definition,
\(
\sum_{i=1}^\ell u(i) X_i
\)
is unitary (has pairwise orthonormal rows) for all unit vectors $u \in \C^\ell$. In particular, every $X_i$ is unitary. For $i \neq j$ we then find that
$X_i+X_j$ is $\sqrt{2}$ times a unitary matrix, so
\[
2I_m =
(X_i + X_j)(X_i + X_j)^H = 2I_m + X_i^{}X_j^H + X_j^{}X_i^H,
\]
that is,
$X_j^{} X_i^H  +  X_i^{}X_j^H = 0$. But also we see that
$X_i+\mathrm{i}X_j$ is also $\sqrt{2}$ times a unitary matrix, so
\[
2I_m = (X_i + \mathrm{i}X_j)(X_i + \mathrm{i} X_j)^H = (X_i + \mathrm{i} X_j)(X_i^H -\mathrm{i}X_j^H)=  2I_m + \mathrm{i}( X_j^{}X_i^H -  X_i^{}X_j^H),
\]
that is, $ X_j^{}X_i^H - X_i^{}X_j^H  = 0$. We conclude that $X_jX_i^H  = 0$ for all $i \neq j$. This would mean that the $\ell$ row spaces of the matrices $X_i$ are pairwise orthogonal subspaces in $\C^n$, but each of dimension $m$. Since $\ell m > n$, this is not possible.
\end{proof}

The above result appears surprising in comparison to the real case. In particular, it admits the following remarkable corollary on a slight variation of the Hurwitz problem. The statement has a classical feel, but since we have been unable to find it in the literature, we emphasize it here. As a matter of fact, our proof of nonexistence of unitary tensors as conducted above resembles the main logic of contradiction in Hurwitz's original proof~\cite{Hurwitz1898}, but under stronger assumptions that rule out all dimensions $n>1$.  The subtle difference to Hurwitz's setup is that the function $u \mapsto \| u \|_2^2$ is not a quadratic form on $\C^n$ over the field $\C$ (it is not $\C$-homogeneous) but is generated by a sesquilinear form.

\begin{corollary}
If $n > 1$, then there exists no bilinear map $\omega \colon \C^n \times \C^n \to \C^n$ such that
\[
\| \omega(u,v) \|_2 = \| u \|_2 \| v \|_2
\]
for all $u,v \in \C^n$.
\end{corollary}

\begin{proof}
Since bilinear forms from $\C^n \times \C^n$ to $\C^n$ are in one-to-one correspondence to complex $n \times n \times n$ tensors via~\eqref{eq:induced d-1 form}, the assertion follows from Theorem~\ref{th: nonexistence in complex case} due to Proposition~\ref{prop:length-preserving}.
\end{proof}

We emphasize again that while unitary tensors do not exist when $\ell m > n$, they do exist when $\ell m\leq n$, by Proposition~\ref{prop: sharpness for tall tensors}.

\subsection{Implications to tensor spaces of order larger than three}

Obviously, it follows from the recursive nature of the definition that orthogonal (resp., unitary) tensors of size $n_1 \times \dots \times n_d \times n_{d+1}$, where $n_1 \le \dots \le n_d \le n_{d+1}$, can exist only if orthogonal (resp., unitary) tensors of size $n_2 \times \dots \times n_{d+1}$ exist. This rules out, for instance, the existence of orthogonal $3 \times 3 \times 3 \times 3$ tensors, and, more generally, the existence of unitary tensors when $n_{d-2} n_{d-1} > n_d$ (cf.~\eqref{eq: nonsharpness for higher-order complex}).

In the real case, the construction of orthogonal $n \times n \times n$ tensors from the multiplication tables~\eqref{eq:complex mult}--\eqref{eq:octonion mult} in section~\ref{sec: Hurwitz problem} is very explicit. The construction can be extended to higher orders as follows.

\begin{theorem}\label{th:higher-order construction}
Let $d \ge 2$, $n \in \{2,4,8\}$, $n_1 \le \dots \le n_d$, and $\bX \in \R^{n_1 \times \dots \times n_d}$ be orthogonal. For any fixed $\mu \in \{1,\dots,d-1\}$ satisfying $n \le n_\mu$, take any $n$ slices $\bX_1,\dots,\bX_n \in \R^{[\mu]}$ from $\bX$ perpendicular to mode $\mu$. Then a real orthogonal tensor of order $d+1$ and size $n_1 \times \dots \times n_{\mu-1} \times n\times n \times n_{\mu+1} \times \dots \times n_d$ can be constructed from the tables~{\upshape(\ref{eq:complex mult}--\ref{eq:octonion mult})}, respectively, using $\bX_k$ instead of $e_k$.
\end{theorem}

The proof is given further below.
\rr{Here, using $\bX_k$ instead of $e_k$ in the $(i,j)$th entry in~\eqref{eq:complex mult}--\eqref{eq:octonion mult}  means constructing a tensor $\bX$ of size $n_1 \times \dots \times n_{\mu-1} \times n\times n \times n_{\mu+1} \times \dots \times n_d$ such that
$\bX(:,\ldots,:,i,j,:,\ldots,:)=\bX_k$.}

As an example, $[10,10,16]$ is an admissible triple by the table~\eqref{table for 10-16}. Hence, by the theorem above, orthogonal tensors of size $8 \times \dots \times 8 \times 10 \times 16$ exist for any number $d-2$ of 8's. So the naive bound~\eqref{eq: trivial lower bound} (which equals $1/\sqrt{10 \cdot 8^{d-2}}$ in this example) for the best rank-one approximation ratio is sharp in $\R^{8 \times \dots \times 8 \times 10 \times 16}$. This is in contrast to the restrictive condition in Proposition~\ref{prop: sharpness for tall tensors}. In particular, in light of Theorem~\ref{th: Hurwitz result}, we have the following immediate corollary of Theorem~\ref{th:higher-order construction}.

\begin{corollary} \label{th:higher-order-cubic}
Real orthogonal $n \times \dots \times n$ tensors of order $d \ge 3$ exist if and only if $n = 1,2,4,8$. Consequently,
\[
\App_d(\R;n,\dots,n) = \frac{1}{\sqrt{n^{d-1}}}
\]
if and only if $n = 1,2,4,8$ for $d \ge 3$. Otherwise, the value of $\App_d(\R;n,\dots,n)$ must be larger.
\end{corollary}

In combination with Proposition~\ref{prop: orthogonality of subtensors}, this corollary implies that lots of orthogonal tensors in low dimensions exist.

\begin{corollary}\label{cor: existence in small dimensions}
  If $\max\limits_{1\le\mu\le d}n_\mu=1,2,4,8$, then orthogonal tensors exist in $\R^{n_1 \times \dots \times n_d}$.
\end{corollary}

\begin{proof}[Proof of Theorem~\ref{th:higher-order construction}]
Without loss of generality, we assume $\mu=1$. Let $\bY \in \R^{n \times n \times n_2 \times \dots \times n_d}$ be a tensor constructed in the way described in the statement  from an orthogonal tensor $\bX$. The slices $\bX_k$ of $\bX$ are then orthogonal tensors of size $n_2 \times \dots \times n_d$. The Frobenius norm of $\bY$ takes the correct value
\[
\| \bY \|_\frob = \sqrt{n^2 \cdot \prod_{\mu = 2}^{d-1} n_\mu}.
\]
According to Theorem~\ref{th: main theorem for orthogonal tensors}(a), we hence have to show that $\| \bY \|_2 = 1$. By~\eqref{eq: trivial lower bound}, it is enough to show $\| \bY \|_2 \le 1$. To do so, let $\omega(u,v) = \bX_0 \times_1 u \times_2 v$ denote the multiplication in the composition algebra $\R^n$, that is, $\bX_0$ is the corresponding multiplication tensor $\bX_\C$, $\bX_{\mathbb{H}}$ or $\bX_{\mathbb{O}}$ from~\eqref{eq:complex mult}--\eqref{eq:octonion mult} depending on the considered value of $n$. Then it holds that
\begin{equation}\label{eq: contraction in first two modes}
\bY \times_1 u \times_2 v = \sum_{k=1}^n \omega_k(u,v) \bX_k.
\end{equation}
Let $\| u \|_2 = \| v \|_2 =  1$. Then, by~\eqref{eq: length preserving}, $\| \omega(u,v) \|_2 = 1$. Further let $\bZ$ be a rank-one tensor in $\R^{n_2 \times \dots \times n_d}$ of Frobenius norm one. By~\eqref{eq: contraction in first two modes} and the Cauchy--Schwarz inequality, it then follows that
\begin{align*}
\abs{\langle \bY, u \otimes v \otimes \bZ \rangle_\frob}^2 &= \left(\sum_{k=1}^n \omega_k(u,v) \langle \bX_k , \bZ \rangle_\frob \right)^2 \le \sum_{k=1}^n \abs{\langle \bX_k , \bZ \rangle_\frob}^2.
\end{align*}
By Proposition~\ref{prop: characterization of spectral norm}, the right expression is bounded by $\| \bX \|_2^2$, which equals one by Theorem~\ref{th: main theorem for orthogonal tensors}(a). This proves $\| \bY \|_2 \le 1$.
\end{proof}

\section{Accurate computation of spectral norm}\label{sec: accurate computation}

In the final section, we present some numerical experiments regarding the computation of the spectral norm. We compare state-of-the-art algorithms implemented in the Tensorlab~\cite{tensorlab} toolbox with our own implementation of an alternating SVD method that has been proposed for more accurate spherical maximization of multilinear forms via two-factor updates. It will be briefly explained in section~\ref{sec: ASVD}.

The summary of algorithms that we used for our numerical results is as follows.

\begin{description}
    \item[\normalfont \texttt{cpd}] This is the standard built-in algorithm for low-rank CP approximation in Tensorlab. To obtain the spectral norm, we use it for computing the best rank-one approximation. Internally, \texttt{cpd} uses certain problem-adapted nonlinear least-squares algorithms~\cite{SoberVanBarelDeLathauwer2013}. When used for rank-one approximation as in our case, the initial rank-one guess $u^1 \otimes \dots \otimes u^d$ is obtained from the truncated higher-order singular value decomposition (HOSVD)~\cite{DeLathauweretal2000a,DeLathauweretal2000b}, that is, $u^\mu$ is computed as a dominant left singular vector of a $\{\mu\}$-matricization ($t = \{\mu\}$ in~\eqref{eq: t-matricization}) of tensor $\bX$. The rank-one tensor obtained in this way is known to be nearly optimal in the sense that $\| \bX - u^1 \otimes \dots \otimes u^d \|_\frob \le \sqrt{d} \| \bX - \bY_1 \|_\frob$, where $\bY_1$ is a best rank-one approximation.
    \item[\normalfont \texttt{cpd} (random)] The same method, but using an option to use a random initial guess $u^1 \otimes \dots \otimes u^d$.
    \item[\normalfont ASVD (random)] Our implementation of the ASVD method using the same random initial guess as \texttt{cpd} (random).
    \item[\normalfont ASVD (\texttt{cpd})] The ASVD method using the result of \texttt{cpd} (random) (which was often better than \texttt{cpd}) as the initial guess, i.e., ASVD is used for further refinement. The improvement in the experiments in sections~\ref{sec: experiment 1}--\ref{sec: experiment 3} is negligible (which indicates rather strong local optimality conditions for the \texttt{cpd} (random) solution), and so results for this method are reported only for random tensors in section~\ref{sec: experiment 4}.
\end{description}

\subsection{The ASVD method}\label{sec: ASVD}
The ASVD method is an iterative method to compute spectral norm and best rank-one approximation of a tensor via~\eqref{eq: def spectral norm}. In contrast to the higher-order power method (which updates one factor at a time), it updates two factors of a current rank-one approximation $u^1 \otimes \dots \otimes u^d$ simultaneously, while fixing the others, in some prescribed order. This strategy was initially proposed in~\cite{DeLathauweretal2000b} (without any numerical experiments) and then given later in more detail in~\cite{friedland2013best}. Update of two factors has also been used in a framework of the maximum block improvement method in~\cite{ChenHeLiZhang2012}. Convergence analysis for this type of method was conducted recently in~\cite{Yangetal2016}.

In our implementation of ASVD the ordering of the updates is overlapping in the sense that we cycle between updates of $(u^1,u^2)$, $(u^2,u^3)$, and so on. Assume that the algorithm tries to update the first two factors $u^1$ and $u^2$ while $u^3,\dots,u^d$ are fixed. To maximize the value $\langle \bX, u^1 \otimes u^2 \otimes \dots \otimes u^d \rangle_\frob$ for $u^1, u^2$ with $\| u^1 \| = \| u^2 \| = 1$, we use the simple fact that
\begin{align*}
   \langle \bX, u^1 \otimes u^2 \otimes \dots \otimes u^d \rangle_\frob &= (u^1)^T (\bX \times_3 u^3 \dots \times_d u^d) u^2.
\end{align*}
Therefore, we can find the maximizer $(u^1, u^2)$ as the top left and right singular vectors of the matrix $\bX \times_3 u^3 \dots \times_d u^d$.

\subsection{Orthogonal tensors}\label{sec: experiment 1}

We start by testing the above methods for the orthogonal tensors \eqref{eq:complex mult}--\eqref{eq:octonion mult}, for which we know that the spectral norm after normalization is $1/n$. The result is shown in Table~\ref{table:orthogonal}: all the methods easily find a best rank-one approximation.
It is worth noting that the computed approximants are not always the same, due to the nonuniqueness described in Remark~\ref{rem:nonuniqueness}.

\begin{table}[t]
    \centering
    \caption{Spectral norm estimations for orthogonal tensors}
    \label{table:orthogonal}
    \begin{tabular}{cccc}
    \hline
        $n$ & \texttt{cpd} & \texttt{cpd} (random) & ASVD (random)  \\
    \hline
2&0.500000&0.500000&0.500000 \\
4&0.250000&0.250000&0.250000 \\
8&0.125000&0.125000&0.125000 \\
    \hline
    \end{tabular}
\end{table}

\subsection{Fourth-order tensors with known spectral norm}\label{sec: experiment 2}

\begin{figure}[t]
    \centering
    \includegraphics[width=.7\linewidth]{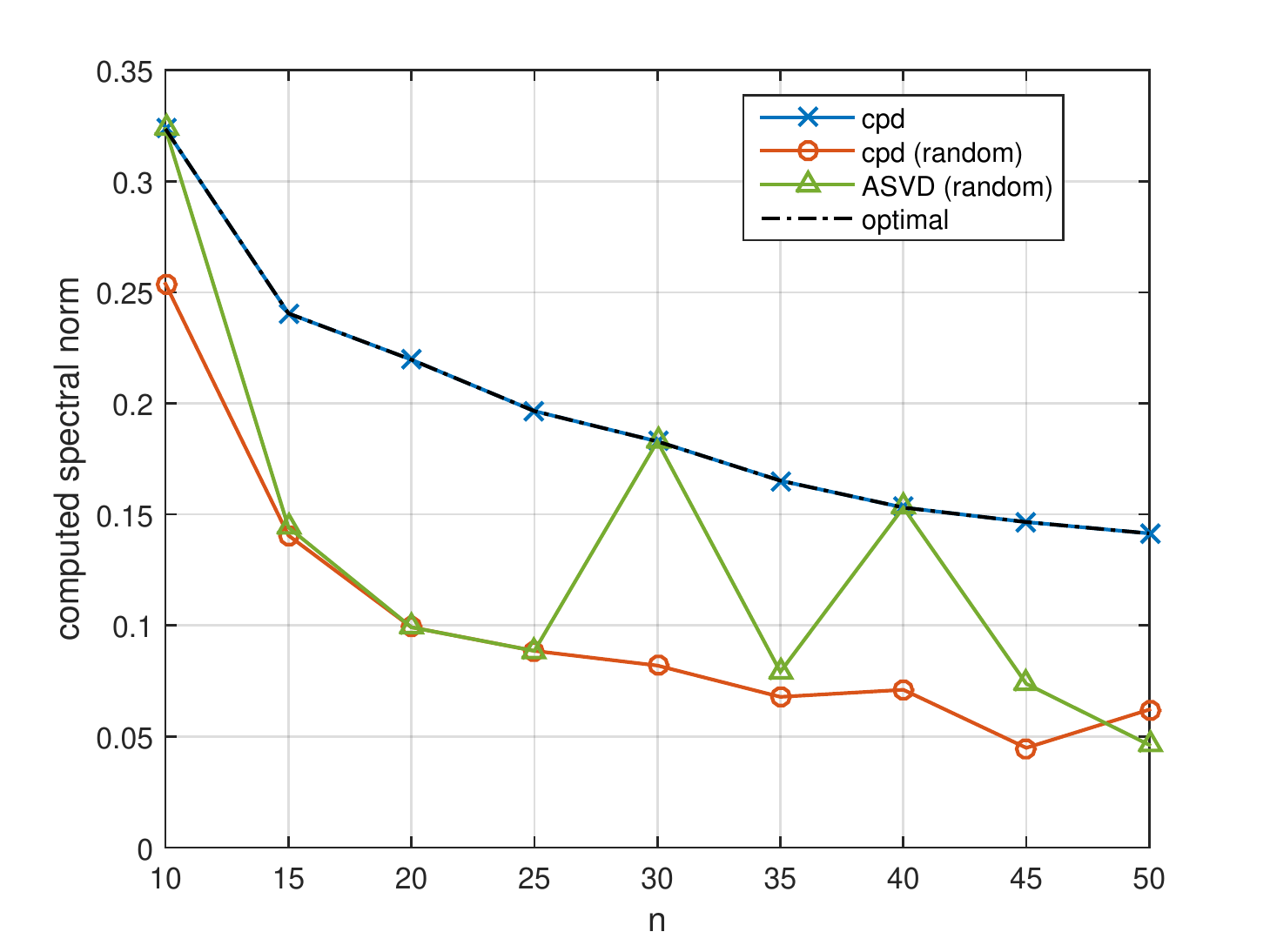}
    \caption{Results for fourth-order tensors with known spectral norms.}
    \label{fig:knownopt}
\end{figure}

In \cite{He2010}, the following examples of fourth-order tensors with known spectral norms are presented. Let
\[
\bX=\sum_{i=1}^m A_i\otimes B_i \mbox{ with } A_i, B_i \in\R^{n\times n} \mbox{ being symmetric},
\]
such that all the eigenvalues of $A_i$ and $B_i$ are in $[-1,1]$, and there are precisely two fixed unit vectors $a,b\in\R^n$ (up to trivial scaling by $-1$)
satisfying
\[
a^{T} A_i a = b^{T} B_i b = 1, \quad i=1,\dots,m.
\]
Clearly, for any unit vectors $x,y,z,w\in\R^{n}$, one has $x^{T}A_iy \le 1 $ and $y^{T}B_iw \le 1$, and so
\[
\langle \bX, x\otimes y\otimes z\otimes w\rangle_\frob \le m = \langle \bX, a\otimes a\otimes b\otimes b\rangle_\frob.
\]
Therefore, $\|\bX\|_2=m$ and $m \cdot a \otimes a \otimes b \otimes b$ is a best rank-one approximation. Moreover, it is not difficult to check that $a$ is the dominant left singular vector of the first ($t = \{1\}$ in~\eqref{eq: t-matricization}) and second ($t = \{2\}$) principal matrix unfolding of $\bX$, while $b$ is the dominant left singular vector of the third and fourth principal matricization. Therefore, for tensors of the considered type, the HOSVD truncated to rank one yields a best rank-one approximation $m \cdot a \otimes a \otimes b \otimes b$.

We construct tensors $\bX$ of this type for $n=10,15,20,\dots,50$ and $m = 10$, normalize them to Frobenius norm one (after normalization the spectral norm is $m/\| \bX \|_\frob$), and apply the considered methods. The results are shown in Figure~\ref{fig:knownopt}. As explained above, the method $\texttt{cpd}$ uses HOSVD for initialization, and indeed it found the optimal factors $a$ and $b$ immediately. Therefore, the corresponding curve in Figure~\ref{fig:knownopt} matches the precise value of the spectral norm. We observe that for most $n$, the methods with random initialization found only suboptimal rank-one approximations.
However, ASVD often found better approximations and in particular found optimal solutions for $n = 10, 30, 40$.

\subsection{Fooling HOSVD initialization}\label{sec: experiment 3}

In the previous experiment the HOSVD truncation yielded the best rank-one approximation. It is possible to construct tensors for which the truncated HOSVD is not a good choice for intialization.

Take, for instance, an $n \times n \times n$ tensor $\bX_n$ with slices
\begin{equation}\label{eq: fooling tensor}
\bX_n(\colon,\colon,k) = S_n^{k-1},
\end{equation}
where $S_n \in \R^{n \times n}$ is the ``shift'' matrix:
\[
    S_n = \begin{bmatrix}
        0 & 1 & &&\\
          & 0 & 1 &&\\
          &   & \ddots & \ddots & \\
          &   & & 0 & 1 \\
        1 &   & &  &0
    \end{bmatrix}.
\]
This tensor has strong orthogonality properties: in any direction, the slices are orthogonal matrices, and parallel slices are pairwise orthogonal in the Frobenius inner product. In particular, $\| \bX_n \|_\frob = n$. However, $\bX_n$ is not an orthogonal tensor in the sense of Definition~\ref{def:ot1}, since $\| \bX_n \|_2 = \sqrt{n}$ (use Proposition~\ref{prop: characterization of spectral norm}). A possible (there are many) best rank-one approximation for $\bX_n$ is given by the ``constant'' tensor whose entries all equal $1/n$. Nevertheless, we observed that the method $\texttt{cpd}$ estimates the spectral norm of $\bX_n$ to be one, which, besides being a considerable underestimation for large $n$, would suggest that this tensor is orthogonal. Figure~\ref{fig:fooling} shows the experimental results for the normalized tensors $\bX_n / n$ and $n = 2,3,\dots,50$.

 \begin{figure}[t]
     \centering
     \includegraphics[width=.7\linewidth]{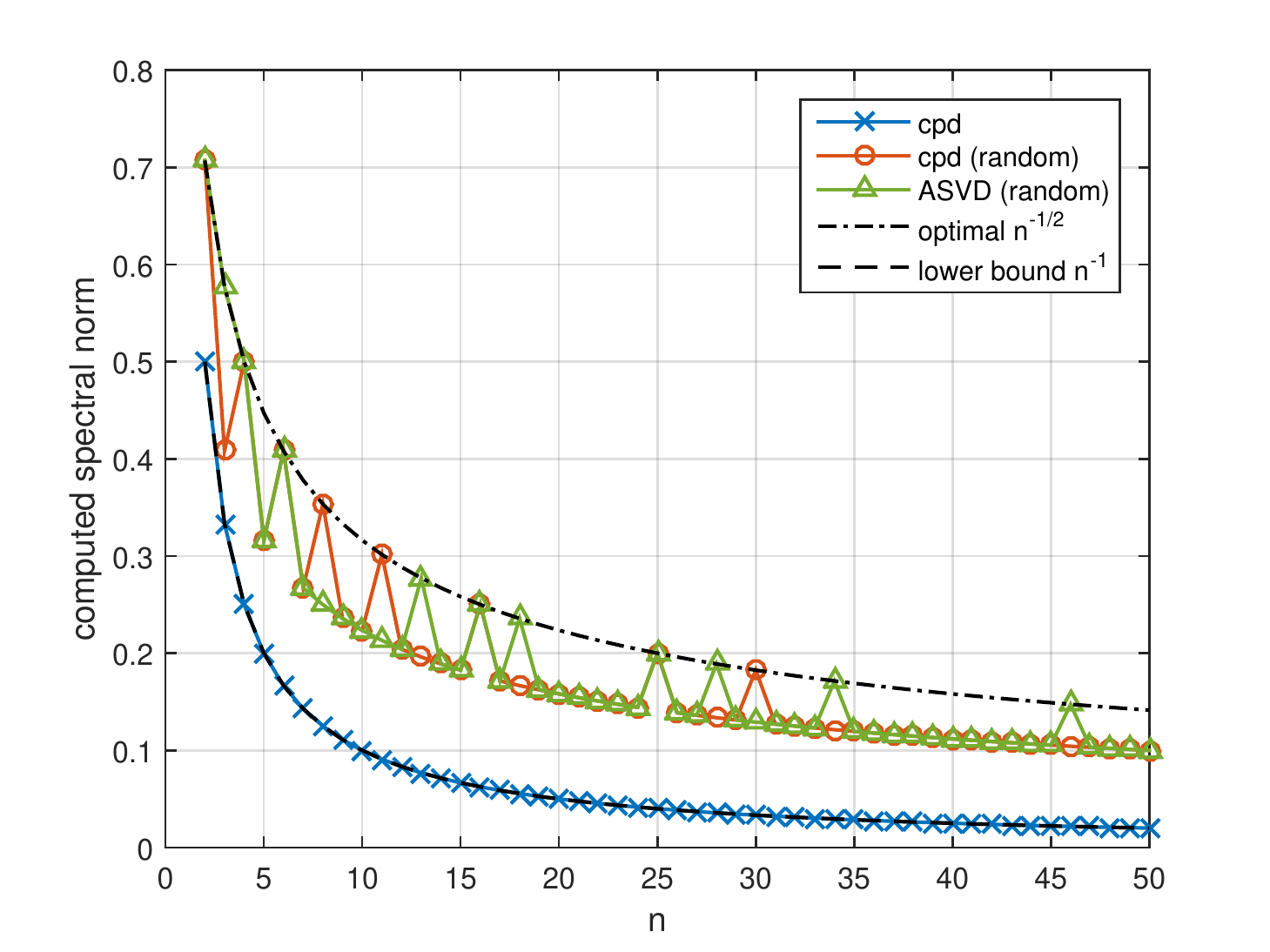}
     \caption{Results for the normalized tensors $\bX_n / n$ from~\eqref{eq: fooling tensor}.
     }
     \label{fig:fooling}
 \end{figure}

The explanation is as follows. The three principal matricization of $\bX_n$ into an $n \times n^2$ matrix all have pairwise orthogonal rows of length $\sqrt{n}$. The left singular vectors are hence just the unit vectors $e_1,\dots,e_n$. Consequently, the truncated HOSVD yields a rank-one tensor $e_i \otimes e_j \otimes e_k$ with $\bX_n(i,j,k) = 1$ as a starting guess. Obviously, $\langle \bX_n, e_i \otimes e_j \otimes e_k \rangle_\frob = 1$. The point is that $e_i \otimes e_j \otimes e_k$ is a critical point for the spherical maximization problem (and thus also for the corresponding rank-one approximation problem~\eqref{eq: best rank-one approximation problem})
\begin{equation}\label{eq: max problem}
\max f(u^1,u^2,u^3) = \langle \bX_n, u^1 \otimes u^2 \otimes u^3 \rangle_\frob \quad \text{s.t.} \quad \| u^1 \|_2 = \| u^2\|_2 = \| u^3 \|_2 = 1.
\end{equation}
To see this, note that $u^1 = e_i$ is the optimal choice for fixed $u^2 = e_j$ and $u^3 = e_k$, since $\bX_n$ has no other nonzero entries in fiber $\bX_n(\colon,j,k)$ except at position $i$. Therefore, the partial derivative $h^1 \mapsto f(h^1,e_j,e_k)$ vanishes with respect to the first spherical constraint, i.e., when $h^1 \perp e_i$ (again, this can be seen directly since such $h^1$ has a zero entry at position $i$). The observation is similar for other directions. As a consequence, $e_i \otimes e_j \otimes e_k$ will be a fixed-point of nonlinear optimization methods for~\eqref{eq: max problem} relying on the gradient or block optimization, thereby providing the function value $f(e_i,e_j,e_k) = 1$ as the spectral norm estimate.

Note that a starting guess $e_i \otimes e_j \otimes e_k$ for computing $\| \bX_n \|_2$ will also fool any reasonable implementation of ASVD. While for, say, fixed $u^3 = e_k$, any rank-one matrix $u^1 \otimes u^2$ of Frobenius norm one will maximize $\langle \bX_n, u^1 \otimes u^2 \otimes e_k \rangle_\frob = (u^1)^T S_n^{k-1} u^2$, its computation via an SVD of $S_n^{k-1}$ will again provide some unit vectors $u^1 = e_i$ and $u^2 = e_j$. We conclude that random starting guesses are crucial in this example. But even then, Figure~\ref{fig:fooling} indicates that there are other suboptimal points of attraction.

\subsection{Spectral norms of random tensors}\label{sec: experiment 4}

Finally, we present some numerical results for random tensors.
In this scenario, Tensorlab's \texttt{cpd} output can be slightly improved using ASVD. Table~\ref{table:vstensorlab-3d} shows the computed spectral norms averaged over 10 samples of real random $20 \times 20 \times 20$ tensors whose entries were drawn from the standard Gaussian distribution.
 Table~\ref{table:vstensorlab-4d} repeats the experiment but with a different size $20 \times 20 \times 20 \times 20$.
In both experiments, ASVD improved the output of \texttt{cpd} in the order of $10^{-3}$ and $10^{-4}$, respectively, yielding the best (averaged) result.

\begin{table}[h]
    \centering
    \caption{Averaged results for random tensors of size $20 \times 20 \times 20$.}
    \label{table:vstensorlab-3d}
    \begin{tabular}{cccc}
        \hline
    \texttt{cpd}  & \texttt{cpd} (random) & ASVD (random) & ASVD (\texttt{cpd}) \\
        \hline
        0.130927 & 0.129384 & 0.129583 & 0.130985 \\
        \hline
    \end{tabular}
\end{table}

\begin{table}[h]
    \centering
    \caption{Averaged results for random tensors of size $20 \times 20 \times 20 \times 20$.}
    \label{table:vstensorlab-4d}
    \begin{tabular}{cccc}
        \hline
    \texttt{cpd}  & \texttt{cpd} (random) & ASVD (random) & ASVD (\texttt{cpd}) \\
        \hline
        0.035697 &0.035265 & 0.034864 & 0.035707
        \\\hline
    \end{tabular}
\end{table}

\begin{figure}[t]
    \centering
    \includegraphics[width=0.7\linewidth]{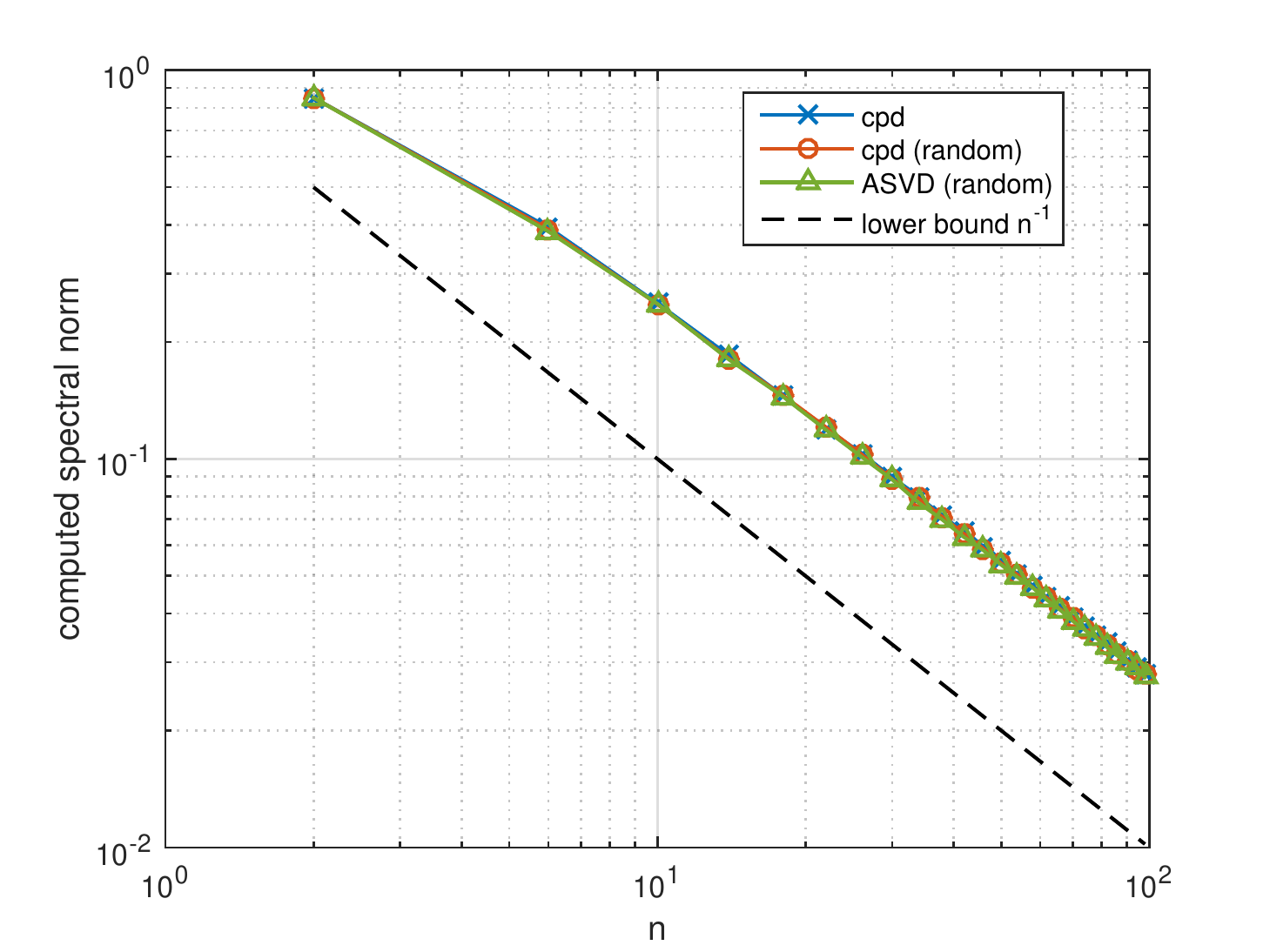}
    \caption{Averaged results for random $n \times n \times n$ tensors.}\label{fig:randomtensor}
\end{figure}

Figure~\ref{fig:randomtensor} shows the averaged spectral norm estimations of real random $n \times n \times n$ tensors for varying $n$ together with the naive lower bound $1/n$ for the best rank-one approximation ratio (we omit the curve for ASVD ({\texttt{cpd}) as it does not look very different from the other ones in the double logarithmic scale). The average is taken over 20 random tensors for each $n$. From Theorem~\ref{th: Hurwitz result} we know that the lower bound is not tight for $n \neq 1, 2, 4, 8$. Nevertheless, we observe an asymptotic order $O(1/n)$ for the spectral norms of random tensors. This illustrates the theoretical results mentioned in section~\ref{sec: random results}. In particular, $\App_3(\R;n, n, n) = O(1/n)$ as explained in section~\ref{sec: random results}; see~\eqref{eq:randombound} and~\eqref{eq: asymptotic behavior nxnxn}.

\section*{Acknowledgments}
The authors are indebted to Jan Draisma, who pointed out the connection between real orthogonal third-order tensors and the Hurwitz problem, and also to Thomas K\"uhn for bringing the valuable references~\cite{CobosKuehnPeetre1992,CobosKuehnPeetre1999,CobosKuehnPeetre2000,KuehnPeetre2006} to our attention.

\bibliographystyle{amsplain}
\bibliography{rank1ratio}

\providecommand{\bysame}{\leavevmode\hbox to3em{\hrulefill}\thinspace}
\providecommand{\MR}{\relax\ifhmode\unskip\space\fi MR }
\providecommand{\MRhref}[2]{%
  \href{http://www.ams.org/mathscinet-getitem?mr=#1}{#2}
}
\providecommand{\href}[2]{#2}
\begin{thebibliography}{10}

\bibitem{ChenHeLiZhang2012}
B.~Chen, S.~He, Z.~Li, and S.~Zhang, \emph{Maximum block improvement and
  polynomial optimization}, SIAM J. Optim. \textbf{22} (2012), no.~1, 87--107.

\bibitem{ChenXuZhu2010}
L.~Chen, A.~Xu, and H.~Zhu, \emph{Computation of the geometric measure of
  entanglement for pure multiqubit states}, Phys. Rev. A \textbf{82} (2010),
  032301.

\bibitem{CobosKuehnPeetre1992}
F.~Cobos, T.~K\"uhn, and J.~Peetre, \emph{Schatten-von {N}eumann classes of
  multilinear forms}, Duke Math. J. \textbf{65} (1992), no.~1, 121--156.

\bibitem{CobosKuehnPeetre1999}
\bysame, \emph{On {${\mathfrak G}_p$}-classes of trilinear forms}, J. London
  Math. Soc. (2) \textbf{59} (1999), no.~3, 1003--1022.

\bibitem{CobosKuehnPeetre2000}
\bysame, \emph{Extreme points of the complex binary trilinear ball}, Stud.
  Math. \textbf{138} (2000), no.~1, 81--92.

\bibitem{DeLathauweretal2000a}
L.~De~Lathauwer, B.~De~Moor, and J.~Vandewalle, \emph{A multilinear singular
  value decomposition}, SIAM J. Matrix Anal. Appl. \textbf{21} (2000), no.~4,
  1253--1278.

\bibitem{DeLathauweretal2000b}
\bysame, \emph{On the best rank-1 and rank-{$(R_1,R_2,\cdots,R_N)$}
  approximation of higher-order tensors}, SIAM J. Matrix Anal. Appl.
  \textbf{21} (2000), no.~4, 1324--1342.

\bibitem{DerksenFriedlandLimWang2017}
H.~Derksen, Friedland S., L.-H. Lim, and L.~Wang, \emph{Theoretical and
  computational aspects of entanglement}, arXiv:1705.07160, 2017.

\bibitem{friedland2013best}
S.~Friedland, V.~Mehrmann, R.~Pajarola, and S.~K. Suter, \emph{On best rank one
  approximation of tensors}, Numer. Linear Algebra Appl. \textbf{20} (2013),
  no.~6, 942--955.

\bibitem{GnangElgammalRetakh2011}
E.K. Gnang, A.~Elgammal, and V.~Retakh, \emph{A spectral theory for tensors},
  Ann. Fac. Sci. Toulouse S\'{e}r \textbf{20} (2011), 801--841.

\bibitem{golubbook4th}
G.H. Golub and C.F. Van~Loan, \emph{{Matrix Computations}}, 4th ed., Johns
  Hopkins University Press, Baltimore, MD, 2013.

\bibitem{GrossFlamiaEisert2009}
D.~Gross, S.~T. Flammia, and J.~Eisert, \emph{Most quantum states are too
  entangled to be useful as computational resources}, Phys. Rev. Lett.
  \textbf{102} (2009), 190501.

\bibitem{He2010}
S.~He, Z.~Li, and S.~Zhang, \emph{Approximation algorithms for homogeneous
  polynomial optimization with quadratic constraints}, Math. Program.
  \textbf{125} (2010), no.~2, Ser. B, 353--383.

\bibitem{HiguchiSudbery2000}
A.~Higuchi and A.~Sudbery, \emph{How entangled can two couples get?}, Phys.
  Lett. A \textbf{273} (2000), no.~4, 213--217.

\bibitem{HornJohnson1985}
R.A. Horn and C.R. Johnson, \emph{{Matrix {A}nalysis}}, Cambridge University
  Press, Cambridge, UK, 1985.

\bibitem{Hurwitz1898}
A.~Hurwitz, \emph{{{\"U}ber die Composition der quadratischen Formen von
  belibig vielen Variablen}}, Nachrichten von der Gesellschaft der
  Wissenschaften zu {G\"ottingen}, Mathematisch-Physikalische Klasse, 1898,
  pp.~309--316.

\bibitem{Hurwitz1922}
\bysame, \emph{\"uber die {K}omposition der quadratischen {F}ormen}, Math. Ann.
  \textbf{88} (1922), no.~1-2, 1--25.

\bibitem{JiangKong2015}
Y.-L. Jiang and X.~Kong, \emph{On the uniqueness and perturbation to the best
  rank-one approximation of a tensor}, SIAM J. Matrix Anal. Appl. \textbf{36}
  (2015), no.~2, 775--792.

\bibitem{Kolda2001}
T.G. Kolda, \emph{Orthogonal tensor decompositions}, SIAM J. Matrix Anal. Appl.
  \textbf{23} (2001), no.~1, 243--255.

\bibitem{KoldaBader2009}
T.G. Kolda and B.W. Bader, \emph{Tensor decompositions and applications}, SIAM
  Rev. \textbf{51} (2009), no.~3, 455--500.

\bibitem{KongMeng2015}
X.~Kong and D.~Meng, \emph{The bounds for the best rank-1 approximation ratio
  of a finite dimensional tensor space}, Pac. J. Optim. \textbf{11} (2015),
  no.~2, 323--337.

\bibitem{KuehnPeetre2006}
T.~K\"uhn and J.~Peetre, \emph{Embedding constants of trilinear {S}chatten-von
  {N}eumann classes}, Proc. Est. Acad. Sci. Phys. Math. \textbf{55} (2006),
  no.~3, 174--181.

\bibitem{Lenzhenetal2011}
A.~Lenzhen, S.~Morier-Genoud, and V.~Ovsienko, \emph{New solutions to the
  {H}urwitz problem on square identities}, J. Pure Appl. Algebra \textbf{215}
  (2011), 2903--2911.

\bibitem{NguyenDrineasTran2015}
N.H. Nguyen, P.~Drineas, and T.D. Tran, \emph{Tensor sparsification via a bound
  on the spectral norm of random tensors}, Inf. Inference \textbf{4} (2015),
  no.~3, 195--229.

\bibitem{parlettsym}
B.N. Parlett, \emph{{The Symmetric Eigenvalue Problem}}, Society for Industrial
  and Applied Mathematics (SIAM), Philadelphia, PA, 1998.

\bibitem{Qi2011}
L.~Qi, \emph{The best rank-one approximation ratio of a tensor space}, SIAM J.
  Matrix Anal. Appl. \textbf{32} (2011), no.~2, 430--442.

\bibitem{Radon1922}
J.~{Radon}, \emph{{Lineare Scharen orthogonaler Matrizen.}}, {Abh. Math. Semin.
  Univ. Hamb.} \textbf{1} (1922), no.~1, 1--14.

\bibitem{Shapiro2000}
D.~Shapiro, \emph{{Compositions of Quadratic Forms}}, Walter de Gruyter Co.,
  Berlin, 2000.

\bibitem{SoberVanBarelDeLathauwer2013}
L.~Sorber, M.~Van~Barel, and L.~De~Lathauwer, \emph{Optimization-based
  algorithms for tensor decompositions: canonical polyadic decomposition,
  decomposition in rank-{$(L_r,L_r,1)$} terms, and a new generalization}, SIAM
  J. Optim. \textbf{23} (2013), no.~2, 695--720.

\bibitem{Tomioka2014}
R.~Tomioka and T.~Suzuki, \emph{Spectral norm of random tensors},
  arXiv:1407.1870, 2014.

\bibitem{Uschmajew2015}
A.~Uschmajew, \emph{Some results concerning rank-one truncated steepest descent
  directions in tensor spaces}, {Proceedings of the International Conference on
  Sampling Theory and Applications}, 2015, pp.~415--419.

\bibitem{tensorlab}
N.~Vervliet, O.~Debals, L.~Sorber, M.~Van~Barel, and L.~De~Lathauwer,
  \emph{Tensorlab v3.0}, March 2016, Available online, Mar. 2016. URL:
  \texttt{http://www.tensorlab.net/}.

\bibitem{Yangetal2016}
Y.~Yang, S.~Hu, L.~De Lathauwer, and J.A.K. Suykens, \emph{Convergence study of
  block singular value maximization methods for rank-1 approximation to higher
  order tensors}, Internal Report 16-149, ESAT-SISTA, KU Leuven (2016),
  \texttt{ftp://ftp.esat.kuleuven.ac.be/pub/stadius/yyang/study.pdf}.

\bibitem{Yiu1994}
P.~Yiu, \emph{Composition of sums of squares with integer coefficients},
  {Deformations of Mathematical Structures II: Hurwitz-Type Structures and
  Applications to Surface Physics. Selected Papers from the Seminar on
  Deformations, {\L}{\'o}d{\'{z}}-Malinka, 1988/92} (J.~{\L}awrynowicz, ed.),
  Springer Netherlands, Dordrecht, 1994, pp.~7--100.

\end{thebibliography}

\end{document}